\documentclass[12pt]{amsart}
\usepackage{amsmath}
\usepackage{mathtools}
\usepackage{amssymb}
\usepackage{amsthm}
\usepackage{mathrsfs}
\usepackage[pagebackref,colorlinks=true,linkcolor=blue,citecolor=blue]{hyperref}
\usepackage{tikz-cd}

\usepackage[english]{babel}
\usepackage{graphicx}
\usepackage{caption}
\usepackage{float}

\usepackage{bbding} 

\UseRawInputEncoding

\tikzset{
  symbol/.style={
    draw=none,
    every to/.append style={
      edge node={node [sloped, allow upside down, auto=false]{$#1$}}}
  }
}

\newcommand{\Z}{\mathbb{Z}}

\newcommand{\R}{\mathbb{R}}
\newcommand{\BC}{\mathbb{C}}

\newcommand{\A}{{\mathbb A}}

\newcommand{\ol}{\overline}

\newcommand{\SL}{\mathrm{SL}}
\newcommand{\GL}{\mathrm{GL}}
\newcommand{\SO}{\mathrm{SO}}
\newcommand{\Sp}{\mathrm{Sp}}

\newcommand{\RO}{\mathrm{O}}

\newcommand{\RG}{\mathrm{G}}

\newcommand{\Supp}{\mathrm{Supp}}
\newcommand{\Sym}{\mathrm{Sym}}

\newcommand{\mfr}[1]{\mathfrak{#1}}

\newcommand{\Eseg}{\underline{\mathrm{Eseg}}}
\newcommand{\VEseg}{\underline{\mathrm{VEseg}}}
\newcommand{\half}[1]{\frac{#1}{2}}

\newcommand{\comment}[1]{}
\newcommand{\EE}{\mathcal{E}}
\newcommand{\EEE}{\mathfrak{E}}
\newcommand{\ee}{\mathfrak{e}}

\newcommand{\NV}{\mathrm{NV}}

\newtheorem{thm}{Theorem}[section]
\newtheorem{cor}[thm]{Corollary}
\newtheorem{lemma}[thm]{Lemma}
\newtheorem{prop}[thm]{Proposition}
\newtheorem {conj}[thm]{Conjecture}

\newtheorem {ques/conj}[thm]{Question/Conjecture}

\newtheorem{defn}[thm]{Definition}
\newtheorem{remark}[thm]{Remark}

\newtheorem{exmp}[thm]{Example}

\newtheorem*{globalcond*}{Global Condition}
\newtheorem*{localcond*}{Local Condition}

\newtheorem*{globalconj*}{Global Conjecture}
\newtheorem*{localconj*}{Local Conjecture}

\newtheorem*{nonzero*}{Conjecture on the non-vanishing of the normalized intertwining operators}
\newtheorem*{holo*}{Conjecture on the holomorphicity of the normalized intertwining operators}

\DeclareMathOperator{\Ind}{Ind}

\numberwithin{equation}{section}

\let\oldbullet\bullet
\renewcommand{\bullet}{{\vcenter{\hbox{\tiny$\oldbullet$}}}}

\begin{document}

\title[Complementary Arthur representations and unitary dual]{On the complementary Arthur representations and unitary dual for p-adic classical groups}

\author[A. Hazeltine]{Alexander Hazeltine
}
\address{Department of Mathematics\\
University of Michigan\\
Ann Arbor, MI 48109, USA}
\email{ahazelti@umich.edu}

\author[D. Jiang]{Dihua Jiang}
\address{School of Mathematics, University of Minnesota, Minneapolis, MN 55455, USA}
\email{dhjiang@math.umn.edu}

\author[B. Liu]{Baiying Liu}
\address{Department of Mathematics\\
Purdue University\\
West Lafayette, IN, 47907, USA}
\email{liu2053@purdue.edu}

\author[C.-H. Lo]{Chi-Heng Lo}
\address{Department of Mathematics\\
National University of Singapore\\
119076, Singapore}
\email{{ch\_lo@nus.edu.sg}}

\author[Q. Zhang]{Qing Zhang}
\address{School of Mathematics and Statistics, Huazhong University of Science and Technology, Wuhan, 430074, China}
\email{qingzh@hust.edu.cn}


\date{\today}


\thanks{The research of the second-named author is partially supported by the NSF Grant DMS-2200890 and the Simons Fellow: SFI-MPS-SFM-00005659. The research of the third-named author is partially supported by the NSF Grant DMS-1848058. The research of the fifth-named author is partially supported by NSFC Grant 12371010.}

\begin{abstract}
In \cite{HJLLZ24}, we proposed a new conjecture on the structure of the unitary dual of connected reductive groups over non-Archimedean local fields of characteristic zero based on their Arthur representations and verified it for all the known cases on the unitary dual problem. One step towards this conjecture involves the question  whether certain complementary Arthur representations are unitary. 
In this paper, we give an explicit characterization of the complementary Arthur representations for symplectic and split odd special orthogonal groups. As applications, we obtain interesting constraints on local components of irreducible self-dual cuspidal automorphic representations of $\GL_N$, especially when $N=2,3$. 
\end{abstract}

\maketitle


\section{Introduction}\label{sec intro}

Let $F$ be a non-Archimedean local field of characteristic zero and $\RG$ be a connected reductive group defined over $F$. Set $G=\RG(F)$ to be the $F$-rational points of $\RG$. 
In \cite{HJLLZ24}, towards the understanding of the unitary dual $\Pi_u(G)$, we proposed a new conjecture on how to construct the whole unitary dual $\Pi_u(G)$ in terms of the Arthur representations of $G$. More precisely, in \cite[Conjecture 6.1]{Art89}, Arthur proposed a conjectural theory on local Arthur packets, which are parameterized by local Arthur parameters $\Psi(G)$ and consist of representations occurring as local components of automorphic representations in the discrete spectrum (hence expected to be unitary) of $G$. Such representations are called the Arthur representations or the representations of Arthur type, denoted by 
$\Pi_A(G)$. Under the process of taking complementary series or unitary parabolic induction, we defined the ``closure" of $\Pi_A(G)$, which is denoted by $\Pi_{\ol{A}}(G)$, and another set $\Pi_{\ol{A}}^{\lim}(G)$ which contains $\Pi_{\ol{A}}(G) $ and its boundary points. Our new conjecture on the unitary dual can be stated as follows. 

\begin{conj}[{\cite[Conjecture 5.12]{HJLLZ24}}]\label{conj A closure}
Assuming the theory of local Arthur packets as in \cite[Conjecture 6.1]{Art89}, the unitary dual $\Pi_u(G)$ can be constructed from the Arthur representations of $G$, i.e., 
    $$\Pi_{\ol{A}}^{\lim}(G)=\Pi_u(G).$$ 
\end{conj}

For quasi-split classical groups $G$, Arthur introduced a certain enlargement of the set of local Arthur parameters denoted by $\Psi^+_{1/2}(G) \supset \Psi(G)$ (see \eqref{lap}). For $\psi\in\Psi^+_{1/2}(G),$ Arthur also defined the corresponding local Arthur packet $\Pi_\psi$. We let $\Pi_{A+}(G):= \cup_{\psi \in \Psi^+_{1/2}(G)} \Pi_\psi$. Arthur also introduced a smaller subset $\Psi^+_{\mathrm{unit}}(G)$ (\cite[\S 8.3]{Art13}, see \eqref{eq def Psi^+_unit}). Without assuming the Ramanujan conjecture, one can only conclude that the local component of a discrete automorphic representation lies in local Arthur packets corresponding to $\Psi^+_{\mathrm{unit}}(G).$ In \cite[Conjecture 8.3.1]{Art13}, Arthur conjectured that the local Arthur packets corresponding to local Arthur parameters in $\Psi^+_{\mathrm{unit}}(G)$ all consist of unitary representations. However, this assertion may not be true in general due to the counterexamples found in \cite[Example 5.1]{HJLLZ24}. Therefore, we introduce the following set:
\[ 
\Pi_{A+,\, u}(G):=\Pi_{A+}(G) \cap \Pi_u(G).
\]
In order to understand Conjecture \ref{conj A closure}, towards the unitary dual problem, it is desirable to characterize explicitly those unitary representations in $\Pi_{A+,\, u}(G)$, which we call the {\it complementary Arthur representations} of $G$. 

In this paper, we give an explicit characterization of $\Pi_{A+,\, u}(G_n)$, for $G_n=\Sp_{2n}(F)$ or $\SO_{2n+1}(F)$, split groups of rank $n$. More precisely, let $\psi\in \Psi^{+}_{\mathrm{unit}}(G_n)$ with decomposition
\[ \psi= \bigoplus_{i \in I_{nu}} (\rho_i \lvert \cdot \rvert^{x_i} \otimes S_{a_i} \otimes S_{b_i}+\rho_i^{\vee} \lvert \cdot \rvert^{-x_i} \otimes S_{a_i} \otimes S_{b_i}) + \psi_{u}, \]
where $0< x_i< \half{1} $ for any $i \in I_{nu}$ and  $\psi_{u} \in \Psi(G_m)$ for some $m \leq n$ (for notation, see \eqref{eq general lap}). Write
\begin{equation}\label{u rho a b}
    u_{\rho_i}(a_i,b_i):= \begin{pmatrix}
\frac{a_i-b_i}{2} & \cdots & \frac{a_i+b_i}{2}-1 \\
\vdots & \ddots & \vdots \\
\frac{-a_i-b_i}{2}+1 & \cdots & \frac{b_i-a_i}{2}
\end{pmatrix}_{\rho_i}, 
\end{equation}
the (unitary) generalized Speh representation (for definition, see \eqref{shifted Speh rep}). 
Then by M{\oe}glin's result (see Theorem \ref{thm red from nu to gp}), every representation $\pi$ in $\Pi_{\psi}$ is of the form (see \S \ref{sub-gl} for the notation of the parabolic induction)
\begin{equation}\label{eq comp rep decomp} 
\pi=\bigtimes_{i \in I_{nu}} u_{\rho_i}(a_i, b_i) \lvert \cdot \rvert^{x_i} \rtimes \pi_u
\end{equation}
for some $\pi_u \in \Pi_{\psi_u}$. 

In the October 2024 version of our arXiv paper \cite{HJLLZ24}, we made the following conjecture on the characterization of $\Pi_{A+,u}(G_n)$. 

    \begin{conj}[Complementary Arthur Representations, {\cite[Conjecture 5.9]{HJLLZ24}}]\label{conj A+,u}
For any $\psi\in \Psi^{+}_{\mathrm{unit}}(G_n)$ with $G_n=\Sp_{2n}(F)$ or $\SO_{2n+1}(F)$,
a representation $\pi\in\Pi_\psi$ as in \eqref{eq comp rep decomp} is unitary if and only if for any $i \in I_{nu}$ such that $u_{\rho_i}(a_i,b_i)\rtimes \pi_u$ is reducible, the set
\[ \{j \in I_{nu}\ | \ \rho_j \cong \rho_i, a_j=a_i,\ b_j=b_i\}\]
has even cardinality.
\end{conj}

The main result of this paper is to prove Conjecture \ref{conj A+,u}.

\begin{thm}[Theorem \ref{thm main}]\label{thm A+,u}
The characterization of complementary Arthur representations as in Conjecture \ref{conj A+,u} holds 
for $G_n=\Sp_{2n}(F)$ or $\SO_{2n+1}(F)$. 
\end{thm}

The proof of Theorem \ref{thm A+,u} consists of two parts: 
\begin{enumerate}
    \item [(i)]The base case that $|I_{nu}|=1$. Namely, $\pi= u_{\rho}(a,b)\lvert\cdot\rvert^s \rtimes \pi_A$, where $\pi_A$ is of Arthur type and $0 < s<\half{1}$.
    \item [(ii)] The reduction to the base case.
\end{enumerate}
To deal with the base case, we use Arthur's local intertwining relations {\cite[Theorem 2.4.1]{Art13}} and a general criterion for proving the non-unitarity (\cite[(RP)(i)]{MT11}, see Lemma \ref{lemma non-unitary}). With these results, Part (i) becomes 
\begin{enumerate}
    \item [(i$'$)] If $u_{\rho}(a,b)\lvert\cdot\rvert^s \rtimes \pi_A$ is reducible at $s=0$, then there exist two direct summands of $u_{\rho}(a,b) \rtimes \pi_A$ that have different characters in Arthur's component group.
\end{enumerate}
Our main novelty to achieve Part (i$'$) and (ii) is to introduce the definition of ``interval" and ``adjacent" on extended multi-segments, combinatorial objects used to construct local Arthur packets, and to study their combinatorial properties. This is done in \S \ref{sec Z-extended segments}.

Conjecture \ref{conj A+,u} for quasi-split even orthogonal groups is already stated in \cite[Conjecture 5.9]{HJLLZ24}. The proof of Theorem \ref{thm A+,u} in this paper will be extended to quasi-split even orthogonal groups in \cite{HLL25}. 
We remark that H. Atobe and A. M{\'i}nguez also proved Theorem \ref{thm A+,u} independently using similar ideas  (\cite{AM25}).

Theorem \ref{thm A+,u} gives one constraint on the local components of $L^2$-automorphic representations of $\Sp_{2n}$ and split $\SO_{2n+1}$ at $p$-adic places without knowing the generalized Ramanujan conjecture (\cite{Sar05}).

\begin{cor}\label{WRC}
    Let $\RG_n=\Sp_{2n}$ or $\SO_{2n+1}$ be split over any number field $k$ and $\sigma$ be any automorphic representation occurring in the discrete spectrum of $\RG_n$. 
    At each local place $\nu<\infty$ of $k$, the $\nu$-local component 
    $\sigma_\nu$ satisfies the constraint given in Theorem \ref{thm A+,u}.
\end{cor}

By the Langlands functorial transfer from $\RG_n$ to $\GL_N$, one may obtain similar local constraints 
as in Corollary \ref{WRC} on the local components of irreducible self-dual cuspidal automorphic representations $\Pi$ of $\GL_N$. 
In this case, we have 
\[
L(s,\Pi\times\Pi)=L(s,\Pi,\Sym^2)L(s,\Pi,\wedge^2).
\]
It follows that either $L(s,\Pi,\Sym^2)$ has a simple pole at $s=1$, in which case $\Pi$ is called {\sl of orthogonal type} or $L(s,\Pi,\wedge^2)$ has a simple pole at $s=1$, in which case $\Pi$ is called {\sl of symplectic type}. It is clear that if $\Pi$ is of symplectic type, then $N=2n$ must be even. 
Since $\Pi$ is cuspidal, the local component $\Pi_v$ of $\Pi$ at any finite local place $v$ of $k$ is generic and unitary. Hence, following \cite{Tad86}, it is of the form:
\begin{align}\label{eq Piv}
    \Pi_v\cong u_{\rho_1}(a_1,1)\lvert \cdot \rvert^{x_1} \times \cdots \times u_{\rho_s}(a_s,1) \lvert \cdot \rvert^{x_s} 
    \times \pi_{t}
    \times u_{\rho_s}(a_s,1) \lvert \cdot \rvert^{-x_s} \times \cdots \times u_{\rho_1}(a_1,1) \lvert \cdot \rvert^{-x_1},
\end{align}
where for each $i$, $\rho_i$ is an irreducible unitary supercuspidal representation of $\GL_{n_i}(k_v)$ and $ 0 < x_i <\half{1}$, and $\pi_t$ is a tempered representation. Then we have the following application of 
Theorem \ref{thm A+,u}.

\begin{thm}[{Theorem \ref{thm Ramanujan}}]\label{thm Ramanujan 1 intro}
    Suppose that $\Pi$ is an irreducible cuspidal automorphic representation of $\GL_N(\A_k)$ of orthogonal type (resp. symplectic type). Let $v<\infty$ and write the $v$-local component $\Pi_v$ of the form \eqref{eq Piv}.
For any $1 \leq i \leq s$, if $u_{\rho_i}(a_i,1)$ is of orthogonal type (resp. symplectic type) and $u_{\rho_i}(a_i,1)$ is not a factor of $\pi_t$, then the set 
\[\{ 1 \leq j \leq s\ | \ \rho_j \cong \rho_i \text{ and }a_j=a_i\}\]
        has even cardinality. 
\end{thm}

\begin{remark}
Applying Theorems \ref{thm Ramanujan 1 intro} to the cases of $N=2, 3$, we obtain the following interesting constraints. 
    \begin{enumerate}
        \item [(a)] If an irreducible cuspidal automorphic representation $\Pi$ of $\GL_2(\A_k)$ is of orthogonal type, then for any local place $v < \infty$, the $v$-local component $\Pi_v$ is tempered.
        \item [(b)] Suppose that $\Pi$ is an irreducible self-dual cuspidal automorphic representation of  $\GL_3(\A_k)$ which at any finite local place $v$ of $k$, 
        the $v$-local component takes the form:  
        \[\Pi_v = \chi \lvert \cdot \rvert^x \times 1 \times \chi \lvert \cdot \rvert^{-x}\]
        for some $0 \leq x<\half{1}$ and $\chi^2=1$, where $1$ is the trivial representation of $\GL_1(k_v)$.
If $\chi \neq 1$, then $x=0$, i.e., $\Pi_v$ is tempered. In particular, for any $v < \infty$, if $\Pi_v$ is ramified, then $\Pi_v$ is tempered.
    \end{enumerate}
We expect that the above statements extend to all Archimedean local places $v$. 
\end{remark}

The following is another application and is, in fact, our original motivation to prove Theorem \ref{thm A+,u}. Because $\Pi_{A+,u}(G)$ consists of unitary representations, it is natural to test whether $\Pi_{A+,u}(G) \subset \Pi_{\overline{A}}^{\lim}(G)$.
 For the groups $G_n = \Sp_{2n}(F)$ and $\SO_{2n+1}(F)$, Theorem \ref{thm A+,u} in fact yields a stronger inclusion:

\begin{thm}[{\cite[Theorem 5.12]{HJLLZ24}}]
For $G_n=\Sp_{2n}(F)$ or $\SO_{2n+1}(F)$, we have
    $$\Pi_{A+,u}(G_n)\subset \Pi_{\ol{A}}(G_n) \subset \Pi_{\overline{A}}^{\lim}(G).$$
\end{thm}


Following is the structure of this paper. We recall some necessary notations and preliminary results in \S\ref{sec-pre}. In \S\ref{sec-LIR}, we recall a general criterion for proving the non-unitarity and Arthur's local intertwining relation, and then state the base case of the main result. The main technical part is \S\ref{sec Z-extended segments}, where we develop certain general combinatorial tools on $\Z$-extended multi-segments that are crucial for our proof. In \S\ref{sec extended multi-segments}, we recall some general results on extended multi-segments and Arthur packets. Our main result is proved in \S\ref{sec-proof}. Finally, \S\ref{sec-Piv} is devoted to the proof of Theorems \ref{thm Ramanujan 1 intro}.

\subsection*{Acknowledgements} 
The authors would like to thank Freydoon Shahidi for the interest and constant support. The authors would like to thank James Arthur and Bin Xu for helpful communications on \cite[Example 5.1]{HJLLZ24}. The authors also thank Alexander Stadler for helpful discussions. Finally, the authors thank the referee for the comments and suggestions.
Theorem \ref{thm A+,u} was announced during the Conference on Algebraic Representation Theory (CART) in December 2024, South Korea. The fourth author would like to thank 
the organizers for the invitation and the hospitality.

\section{Notation and Preliminaries} \label{sec-pre}

Let $F$ be a non-Archimedean local field of characteristic zero. Let $n$ be a positive integer. We work with the general linear group $\GL_n(F)$ and classical groups $G_n = \RG_n(F)$, where $\RG_n$ denotes either the symplectic group $\Sp_{2n}$ or the split special orthogonal group $\SO_{2n+1}$. 
Denote by $\Pi(G)$ the admissible dual of $G=\RG(F)$, the set of equivalence classes of irreducible admissible representations of $G$.

\subsection{Parabolic induction and generalized Speh representations}{\label{sub-gl}}

In this subsection, we give the notation for parabolic induction for $\GL_n(F)$ and $G_n= \Sp_{2n}(F)$ or $\SO_{2n+1}(F)$, and define the generalized Speh representations.

We start with $\GL_n(F)$. Fix a Borel subgroup of $\GL_n(F)$. For a standard parabolic subgroup $P$ with Levi factor $M \cong \prod_{i=1}^r \GL_{n_i}(F)$ (where $\sum n_i = n$) and smooth representations $\tau_i$  of $\GL_{n_i}(F)$, we denote the normalized parabolic induction by
\[
\tau_1 \times \cdots \times \tau_r := \Ind_P^{\GL_n(F)} (\tau_1 \boxtimes \cdots \boxtimes \tau_r).
\]

Let $\rho$ be an irreducible unitary supercuspidal representation of $\GL_d(F)$. For $x,y \in \R$ with $x - y \in \Z_{\geq 0}$, we define:
\begin{itemize}
\item The \emph{segment} $[x,y]_\rho := \{\rho\lvert \cdot \rvert^x, \rho\lvert \cdot \rvert^{x-1}, \ldots, \rho\lvert \cdot \rvert^y\}$, where $\lvert \cdot \rvert$ is the composition of $F^\times$'s normalized absolute value with the determinant character;
\item The \emph{Steinberg representation} $\Delta_\rho[x,y]$, which is the unique irreducible subrepresentation of $\rho\lvert \cdot \rvert^x \times \cdots \times \rho\lvert \cdot \rvert^y$.
\end{itemize}
For completeness, we declare $\Delta_\rho[x,x+1]$ (with $y = x+1$) to be the trivial representation of $\GL_0(F)$. 

The Langlands classification provides a complete description of the admissible dual $\Pi(\GL_n(F))$. Each irreducible admissible representation $\tau$ of $\GL_n(F)$ can be uniquely expressed as the irreducible subrepresentation of a standard parabolic induction:
\[
\tau \hookrightarrow \Delta_{\rho_1}[x_1,y_1] \times \cdots \times \Delta_{\rho_r}[x_r,y_r],
\]
where:
\begin{itemize}
    \item each $\rho_i$ is an irreducible unitary supercuspidal representation of $\GL_{n_i}(F)$,
    \item each $[x_i,y_i]_{\rho_i}$ is a segment with $x_i + y_i \leq x_{i+1} + y_{i+1}$ for all $i$.
\end{itemize}
We denote this representation by
\[
\tau = L(\Delta_{\rho_1}[x_1,y_1], \dots, \Delta_{\rho_r}[x_r,y_r]).
\]

Given integers $s,t \geq 1$ and a base point $x_{1,1} \in \R$, define the array:
\[
x_{i,j} = x_{1,1} - i + j \quad \text{for } 1 \leq i \leq s, 1 \leq j \leq t.
\]
The associated \emph{generalized Speh representation} is defined as:
\begin{equation}\label{shifted Speh rep}
   \begin{pmatrix}
x_{1,1} & \cdots & x_{1,t} \\
\vdots & \ddots & \vdots \\
x_{s,1} & \cdots & x_{s,t}
\end{pmatrix}_{\!\rho} 
:= L\left(\Delta_\rho[x_{1,1},x_{s,1}], \dots, \Delta_\rho[x_{1,t},x_{s,t}]\right). 
\end{equation}

Next, we give notation for the parabolic induction for the classical group $G_n$. Fix a Borel subgroup of $G_n$ and let $P$ be a standard parabolic subgroup with Levi decomposition
\[
M \cong \prod_{i=1}^r \GL_{n_i}(F) \times G_m.
\]
For smooth representations $\tau_i$ of $\GL_{n_i}(F)$ ($1 \leq i \leq r$) and $\sigma$ of $G_m$, we denote the normalized parabolic induction by
\[
\tau_1 \times \cdots \times \tau_r \rtimes \sigma := \Ind_P^{G_n}(\tau_1 \boxtimes \cdots \boxtimes \tau_r \boxtimes \sigma).
\]

\subsection{Arthur representations and reduction to good parity}\label{local A packets}

Recall that we have $\RG_n=\Sp_{2n}$ or split $\SO_{2n+1}$ and $G_n = \RG_n(F)$. Let $W_F$ be the Weil group and ${\RG}^\vee_n(\mathbb{C})$ be the Langlands dual group. 
The local Arthur packets $\Pi_{\psi}$ defined in \cite[Theorem 2.2.1]{Art13} are certain finite subsets of $\Pi(G_n)$, satisfying certain twisted endoscopic character identities, and are parameterized by local Arthur parameters
$$\psi: W_F \times \SL_2(\mathbb{C}) \times \SL_2(\mathbb{C}) \rightarrow {{\RG}^\vee_n(\mathbb{C})},$$
\begin{equation}\label{lap}
  \psi = \bigoplus_{i=1}^r \phi_i\lvert \cdot \rvert^{x_i} \otimes S_{a_i} \otimes S_{b_i},  
\end{equation}
satisfying the following conditions: 
\begin{enumerate}
    \item [(1)]$\phi_i(W_F)$ is bounded and consists of semi-simple elements, and $\dim(\phi_i)=d_i$;
    \item [(2)] $x_i \in \R$ and $|x_i|<\half{1}$;
    \item [(3)]the restrictions of $\psi$ to the two copies of $\SL_2(\mathbb{C})$ are algebraic, $S_k$ is the $k$-dimensional irreducible representation of $\SL_2(\mathbb{C})$, and 
    $$\sum_{i=1}^r d_ia_ib_i = N:= 
\begin{cases}
2n+1 & \text{ when } G_n=\Sp_{2n}(F),\\
2n & \text{ when } G_n=\SO_{2n+1}(F).
\end{cases}
$$ 
\end{enumerate}
We let $\Psi^{+}_{1/2}(G_n)$ be the set of local Arthur parameters of $G_n$ and $\Psi(G_n)$ be the subset of $\Psi^+_{1/2}(G_n)$ consisting of local Arthur parameters $\psi$ whose restriction to $W_F$ is bounded. In other words, $\psi \in \Psi(G_n)$ if and only if $x_i=0$ for $i=1,\dots, r$ in the decomposition \eqref{lap}. 

In this paper, a representation $\pi\in\Pi(G_n)$ is called an {\it Arthur representation} or {\it of Arthur type} if $\pi\in\Pi_\psi$ for some local Arthur parameter $\psi \in \Psi(G_n)$. Define
\begin{align*}
    \Pi_{A}(G_n)=\{\pi \in \Pi_{\psi} \ | \ \psi \in \Psi(G_n)\}\quad  {\rm ~and}\quad 
    \Pi_{A+}(G_n)=\{\pi \in \Pi_{\psi} \ | \ \psi \in \Psi^+_{1/2}(G_n)\}.
\end{align*}

Now we recall the reduction of the construction of local Arthur packets to the good parity case. By the Local Langlands Correspondence for $\GL_{d_i}(F)$, a bounded representation $\phi$ of $W_F$ can be identified with an irreducible unitary supercuspidal representation $\rho$ of $\GL_{d_i}(F)$ (\cite{Hen00, HT01, Sch13}). Consequently, we may write \eqref{lap} as
\begin{equation}\label{A-param decomp}
  \psi = \bigoplus_{i \in I } \rho_i\lvert \cdot \rvert^{x_i} \otimes S_{a_i} \otimes S_{b_i},  
\end{equation}
where $\rho_i$'s are unitary supercuspidal representations of $\GL_{d_i}(F)$. 

Under this decomposition, we say a summand $\rho_i\lvert \cdot \rvert \otimes S_{a_i} \otimes S_{b_i}$ of $\psi$ is of \emph{good parity} if $x_i=0$, $\rho_i$ is self-dual, and
\begin{itemize}
    \item $\rho_i$ is of orthogonal type if $G_n=\SO_{2n+1}$ and $a_i+b_i$ is odd, or if $G_n=\Sp_{2n}$ and $a_i+b_i$ is even; and
    \item $\rho_i$ is of symplectic type if $G_n=\SO_{2n+1}$ and $a_i+b_i$ is even, or if $G_n=\Sp_{2n}$ and $a_i+b_i$ is odd.
\end{itemize}
Here we say a self-dual supercuspidal representation of $\GL_d(F)$ is of orthogonal type (resp. symplectic type) if the image of its $L$-parameter $\phi: W_F \to \GL_d(\BC)$ preserves a non-degenerate symmetric (resp. skew-symmetric) bilinear form on $\BC^d$. Finally, we say $\psi$ is of \emph{good parity} if every summand in the decomposition \eqref{A-param decomp} is of good parity.

Let $\psi \in \Psi^{+}_{1/2}(G_n).$ Since $\psi$ factors through $\RG_n^\vee(\BC)$, we may rewrite the decomposition \eqref{A-param decomp} as
\begin{align}\label{eq general lap}
    \psi= & \bigoplus_{i \in I_{nu}} (\rho_i \lvert \cdot \rvert^{x_i} \otimes S_{a_i}\otimes S_{b_i} + \rho_i^{\vee} \lvert \cdot \rvert^{-x_i} \otimes S_{a_i}\otimes S_{b_i})\\
    &\oplus\bigoplus_{i \in I_{bp}} (\rho_i \otimes S_{a_i} \otimes S_{b_i} +\rho_i^{\vee} \otimes S_{a_i} \otimes S_{b_i}) \oplus\bigoplus_{i \in I_{gp}} \rho_i \otimes S_{a_i} \otimes S_{b_i},\nonumber
\end{align}
where 
\begin{enumerate}
    \item [$\oldbullet$] for any $i \in I_{nu}$, $x_i>0$;
    \item [$\oldbullet$] for any $i \in I_{bp}$, $x_i=0$ and $\rho_i\otimes S_{a_i} \otimes S_{b_i}$ is not of good parity;
    \item [$\oldbullet$] for any $i \in I_{gp}$, $x_i=0$ and $\rho_i\otimes S_{a_i} \otimes S_{b_i}$ is of good parity.
\end{enumerate}
Set $x_i:=0$ for $i \in I_{bp}\sqcup I_{gp}$. For $\ast \in  \{ nu, \ bp,\ gp  \}$, define subrepresentations $\psi_{\ast}$ of $\psi$ by 
\[ \psi_{\ast}:= \bigoplus_{i \in I_{\ast}} \rho_i\lvert \cdot \rvert^{x_i} \otimes S_{a_i} \otimes S_{b_i}.\]
Thus, $\psi_{gp}$ is a local Arthur parameter of $G_{m'}$ of good parity for some $m'\leq n$, and 
\begin{align}\label{eq decomp red to gp}
    \psi=( \psi_{nu}+ \psi_{nu}^{\vee}) +(\psi_{bp} + \psi_{bp}^{\vee})+ \psi_{gp}.
\end{align}
Also set $\psi_{u}:= (\psi_{bp}+\psi_{bp}^{\vee}) + \psi_{gp}$. Then $\psi_{u}\in \Psi(G_m)$ for some $m \leq n$.

For each $i \in I_{nu} \sqcup I_{bp}$, define $\tau_i:= u_{\rho_i}(a_i,b_i) \lvert \cdot \rvert^{x_i}$ (see \eqref{u rho a b}), which is the unique irreducible representation in the local Arthur packet $\Pi_{\rho_i \lvert \cdot \rvert^{x_i} \otimes S_{a_i}\otimes S_{b_i}}$ of $\GL_n(F).$ We set 
\[ \tau_{\psi_{nu}}:=\bigtimes_{i \in I_{nu}} \tau_i\quad {\rm and}\quad \tau_{\psi_{bp}}:= \bigtimes_{i \in I_{bp}} \tau_i,\]
which are irreducible and the order of the products does not matter. We recall a smaller subset $ \Psi_{\mathrm{unit}}^{+}(G_n) \subset \Psi^+_{1/2}(G_n)$ defined by
\begin{align}\label{eq def Psi^+_unit}
    \Psi_{\mathrm{unit}}^{+}(G_n):= \{ \psi\in \Psi^+_{1/2}(G_n)  \ | \ \bigtimes_{i \in I_{nu}}  \tau_i \times \tau_i^{\vee} \text{ is unitary} \}.
\end{align}
Based on the classification of the unitary dual of any general linear group in \cite{Tad86}, this definition is equivalent to $\widetilde{\Psi}^+_{\mathrm{unit}}(G_n)$ in \cite[\S 1.5]{Art13}, and is equivalent to the collection of $\psi \in \Psi^+_{1/2}(G_n)$ such that the index set $I_{nu}$ can be equipped with a duality map
\begin{align*}
    (\cdot)^{\vee}: I_{nu} &\to I_{nu},\\
    i & \mapsto i^{\vee}
\end{align*}
satisfying that $a_{i^{\vee}}=a_i$, $b_{i^{\vee}}= b_i$, $x_i=x_{i^{\vee}}$ and $\rho_{i^{\vee}} \cong (\rho_i)^{\vee}$.

For $\psi \in \Psi^{+}_{1/2}(G_n) \backslash \Psi(G_n)$ (i.e. $I_{nu}$ is non-empty), 
the local Arthur packet $\Pi_{\psi}$ is defined to be the set of representations (\cite[(1.5.1)]{Art13})
\begin{align}\label{eq def packet A+}
    \Pi_{\psi}:=\{ \tau_{nu,>0} \rtimes \pi_{u}\ | \ \pi_u \in \Pi_{\psi_u}   \}.
\end{align}
 Indeed, Arthur only defined $\Pi_{\psi}$ for $\psi \in \Psi^{+}_{\mathrm{unit}}(G_n)$, but the definition can be extended naturally for $\psi \in \Psi^{+}_{1/2}(G_n)$. In \cite{Moe11b}, M{\oe}glin showed that for $\psi \in \Psi^+_{1/2}(G_n)$, the parabolic inductions in the right hand side of \eqref{eq def packet A+} are all irreducible. Moreover, $\pi_u$ can be further written as a parabolic induction according to the decomposition of $\psi_{u}$.

\begin{thm}[{\cite[Proposition 5.1]{Moe11b}}]\label{thm red from nu to gp}
Let $\psi\in\Psi^{+}_{1/2}(G_n)$ with decomposition \eqref{eq decomp red to gp}. Then, for any $\pi_{gp}\in\Pi_{\psi_{gp}},$ the parabolic induction $\tau_{\psi_{nu}}\times\tau_{\psi_{bp}}\rtimes\pi_{gp}$ is irreducible. As a consequence, \begin{equation}\label{non-unitary A-packet}
    \Pi_\psi=\{\tau_{\psi_{nu}}\times\tau_{\psi_{bp}}\rtimes\pi_{gp}  | ~ \pi_{gp}\in\Pi_{\psi_{gp}}\}.
\end{equation}
\end{thm}

\subsection{Component groups and its combinatorial description}\label{sec character}
Let $\psi$ be a local Arthur parameter of $G_n$. The component group $\mathcal{S}_{\psi}$ is defined as
\[ \mathcal{S}_{\psi}:= \textrm{Cent}(\textrm{im} (\psi), \RG_n^{\vee}(\BC))/\textrm{Cent}(\textrm{im} (\psi), \RG_n^{\vee}(\BC))^{\circ}  Z(\RG_n^{\vee}(\BC))^{\Gamma}.\]
We let $\widehat{\mathcal{S}}_{\psi} $ be its Pontryagin dual. Namely, the set of irreducible representations of $\mathcal{S}_{\psi}$. In the case of classical groups, $\mathcal{S}_{\psi}$ is always a finite abelian $2$-group. Hence, $\widehat{\mathcal{S}}_{\psi} $ is simply the set of its characters. An important ingredient in the definition of a local Arthur packet is a map
\begin{align*}
    \Pi_{\psi} &\to \widehat{\mathcal{S}}_{\psi},\\
    \pi& \mapsto \langle \cdot, \pi \rangle_{\psi},
\end{align*}
which is neither injective nor surjective in general.

We now recall a combinatorial description of these finite abelian groups. Given a local Arthur parameter $\psi$ with decomposition
\[
\psi = \bigoplus_{i \in I} \rho_i \otimes S_{a_i} \otimes S_{b_i},
\]
let $I_{gp}$ be the subset of $I$ such that $\rho_i \otimes S_{a_i} \otimes S_{b_i}$ is of good parity. We let 
\[ A_{\psi}:= \bigoplus_{i \in I_{gp}} \Z/2\Z e_i,\]
a $\Z/2\Z$-vector space of dimension $|I_{gp}|$, or an abelian $2$-group of size $ 2^{|I_{gp}|}$. Let 
\begin{itemize}
    \item $A_{\psi}^0$ be the subgroup generated by $e_i+e_j$ for any $(i,j) \in I_{gp}^2$ such that $\rho_i\otimes S_{a_i}\otimes S_{b_i} \cong \rho_j\otimes S_{a_j}\otimes S_{b_j}$.
    \item Let $z_{\psi}:= \sum_{i \in I_{gp}} e_i.$
\end{itemize}
Then we have an isomorphism
\[ \mathcal{S}_{\psi}\cong A_{\psi}/ (A_{\psi}^0 + z_{\psi}).\]
The character group $\widehat{\mathcal{S}}_\psi$ is the dual $\Z/2\Z$-vector space of $\mathcal{S}_{\psi}$, which is in bijection with the set of functions 
\[
\varepsilon \colon I_{gp} \to \{-1,1\}
\]
satisfying the following conditions:
\begin{enumerate}
    \item $\varepsilon(i) = \varepsilon(j)$ whenever $\rho_i \otimes S_{a_i} \otimes S_{b_i} \cong \rho_j \otimes S_{a_j} \otimes S_{b_j}$;
    \item $\prod_{\substack{i \in I \\ \varepsilon(i) \neq 0}} \varepsilon(i) = 1$ (product over all nonzero values).
\end{enumerate}
For $e= \sum_{i \in I_{gp}} a_i e_i + (A_{\psi}^0 +z_{\psi}) \in \mathcal{S}_{\psi}$ and $\varepsilon \in \widehat{\mathcal{S}}_{\psi}$, we have
\[ \varepsilon(e)= \prod_{i \in I_{gp}}(-1)^{a_i} \varepsilon(i).\]
For more detailed discussion of the above combinatorial description, we refer to \cite[\S 4]{GGP12}.

\begin{remark}
     By condition (2), we may unambiguously write 
    \[
    \varepsilon(\rho_i \otimes S_{a_i} \otimes S_{b_i}) := \varepsilon(i)
    \]
    for the value on isomorphic summands.
\end{remark}

\section{A criterion on non-unitarity and intertwining operators}\label{sec-LIR}

In this section, we recall a criterion of Mui{\'c} and Tadi{\'c} which determines whether a certain parabolic induction is not unitary (\cite[(RP)(i)]{MT11}, see Lemma \ref{lemma non-unitary}). Then, we briefly recall Arthur's normalized intertwining operator (\cite[(2.4.4)]{Art13}) and use the criterion along with Arthur's local intertwining relation (\cite[Theorem 2.4.1]{Art13}) to conclude that certain parabolic inductions are not unitary (Corollary \ref{cor base case non-unitary}) under a certain assumption (Theorem \ref{thm base case reducible components}) which will be proved in \S \ref{sec-proof}. 

\subsection{The main strategy on proving non-unitarity}

We begin by recalling a non-unitary criterion from \cite{MT11}, specialized to the case under consideration.

Let $M= \GL_d(F) \times G_{n_0}$ be a Levi subgroup of $G_n$ and consider a local Arthur parameter $\psi_M$ of $M$:
\[\psi_M= ( \rho\otimes S_{a}\otimes S_b )  \boxtimes \psi_{0},\]
where 
\begin{itemize}
    \item $\rho \cong \rho^{\vee}$ and $d= \dim(\rho)ab$, and
    \item $\psi_{0}$ is a local Arthur parameter of $G_{n_0}$.
\end{itemize}
Let $\psi$ be the local Arthur parameter of $G_{n}$ by composing $\psi_M$ with the embedding ${}^L M \hookrightarrow {}^L G$. Thus,
\[\psi=  (\rho\otimes S_{a}\otimes S_b )^{\oplus 2} \oplus \psi_0.\]

Under this setting, $W(M)=N_G(M)/M$ consists of exactly two elements, and we let $w$ be the non-trivial one. Let $\pi_M= u_{\rho}(a,b) \otimes \pi_{0}$ be an irreducible representation in $\Pi_{\psi_M}$. We let $N(0)$ be the normalized intertwining operator
\[\langle \widetilde{u},\widetilde{\pi}_M\rangle R_{P}(w_u, \widetilde{\pi}, \psi):\Ind_P^{G_n} (u_{\rho}(a,b)\otimes \pi_{0})\to \Ind_P^{G_n} (u_{\rho}(a,b) \otimes \pi_{0}) \]
defined in \cite[(2.4.4)]{Art13}, where $u$ is chosen such that $w_u=w$. We refer the reader to \cite[\S 2.3]{Art13} or \cite[\S 1.7, 1.10]{AGIKMS24} for the precise definition. We can extend it to a family of intertwining operators 
\[N(s): \Ind_P^{G_n} (u_{\rho}(a,b)\lvert \cdot \rvert^{s} \otimes \pi_{0})\to \Ind_P^{G_n} (u_{\rho}(a,b)\lvert \cdot \rvert^{-s} \otimes \pi_{0}). \]

Let
\[\Pi_s:= \Ind_{P}^G u_{\rho}(a,b)\lvert \cdot \rvert^s \otimes \pi_{0}, \]
which is irreducible for $0<s<\half{1}$ by Theorem \ref{thm red from nu to gp}. 
If $\Pi_0=\Ind_{P}^{G_n} \pi_M $ is irreducible, then clearly $\{ \Pi_s\ | \ 0 \leq s <\half{1}\}$ is a continuous family of irreducible Hermitian representations with $\Pi_0$ unitary. Hence, $\{ \Pi_s \ | \ 0 < s<\half{1}\}$ are all unitary.

On the other hand, if $\Pi_0$ is reducible, then we recall a criterion of Mui{\'c} and Tadi{\'c} on the non-unitarity of $\{ \Pi_s\ | \ 0 < s <\half{1}\}$.

\begin{lemma}[{\cite[(RP)(i)]{MT11}}]\label{lemma non-unitary}
   Suppose that $\Pi_0=\mathrm{Ind}_P^{G_n}(\pi_M)$ is reducible and $N(0)$ is not a scalar. Then $\Pi_s$ is not unitary for $0<s < \half{1}.$
\end{lemma}

\subsection{Local intertwining relation}\label{subsec LIR}

In this subsection, we recall the local intertwining relation (\cite[Theorem 2.4.1]{Art13}) in our setting, which describes the action of the intertwining operator $N(0)$ on $\Pi_0=\Ind_P^{G_n} (u_{\rho}(a,b) \otimes \pi_{0})$. We also refer to \cite[\S 1.7]{AGIKMS24} for more details.

First, recall that there is an injection between the component groups, which induces a surjection map between their Pontryagin duals via restriction:
\[ \iota:\mathcal{S}_{\psi_{0}} \cong \mathcal{S}_{\psi_M} \hookrightarrow   \mathcal{S}_{\psi},\ \ \ \iota^{\ast}:  \widehat{\mathcal{S}}_{\psi}  \twoheadrightarrow \widehat{\mathcal{S}}_{\psi_{0}}.\]
With the combinatorial description of these component groups, we see that the above maps are isomorphisms unless $\rho\otimes S_{a}\otimes S_{b}$ is of good parity and $\psi_{0}$ does not contain any summand isomorphic to $\rho\otimes S_{a}\otimes S_{b}$.

Fixing $\varepsilon_{0} \in  \widehat{\mathcal{S}}_{\psi_{0}},$ it is known that
\begin{align}\label{eq unitary induction Arthur}
    \bigoplus_{\substack{\pi_{0} \in \Pi_{\psi_{0}},\\ \langle \cdot, \pi_{0} \rangle_{\psi}= \varepsilon_{0}}} u_{\rho}(a,b) \rtimes \pi_{0}= \bigoplus_{\substack{\pi \in \Pi_{\psi}, \\ \iota^{\ast} (\langle \cdot, \pi \rangle_{\psi})= \varepsilon_{0} }} \pi.
\end{align}
See \cite[Proposition 4.2]{Ato22d}. An explicit decomposition of an individual induction $u_{\rho}(a,b) \rtimes \pi_{0}$ is described in \cite{Ato22d} using extended multi-segments, which we recall in \S \ref{sec unitary induction of Arthur type}. In particular, \eqref{eq unitary induction Arthur} implies that
\[ \bigoplus_{\pi_M \in \Pi_{\psi_0}} \Ind_{P}^G \pi_M = \bigoplus_{\pi \in \Pi_{\psi}} \pi. \]
It is known that local Arthur packets are multiplicity free (\cite{Moe11a}). Hence any induction $\Pi_0$ is multiplicity free, and the intertwining operator $N(0)$ acts on each irreducible subquotient $\pi$ of $\Pi_0$ by a scalar, which we denote by $c(\pi)$.

Let $s_u=e + (A_{\psi}^0 +z_{\psi}) $ in $\mathcal{S}_{\psi}$, where $e$ corresponds to a summand $\rho\otimes S_{a}\otimes S_b$ inside $\psi$ (see \S \ref{sec character}). Note that $s_u$ generates the quotient $\mathcal{S}_{\psi}/ \iota(\mathcal{S}_{\psi_M})$. The local intertwining relation (\cite[Theorem 2.4.1]{Art13}) is the equation in the Grothendieck group
\begin{align}\label{eq A-LIR}
 \tag{A-LIR} \sum_{\pi \in \Pi_{\psi}} \langle   s_u, \pi \rangle_{\psi} \pi= \sum_{\pi_{M} \in \Pi_{\psi_{M}}}  \sum_{\pi \leq \Ind_P^G \pi_M} c(\pi) \pi= \sum_{\pi \in \Pi_{\psi}} c(\pi) \pi.
\end{align}
Again, since $\Pi_{\psi}$ is multiplicity free, we see that $c(\pi)= \langle s_u , \pi \rangle_{\psi}$. This is the (LIR) in \cite[\S 1.10]{AGIKMS24}.

Later, we shall prove the following theorem, which is implied by Theorem \ref{thm reducibility unitary induction}.

\begin{thm}\label{thm base case reducible components}
Let $\pi_{0}\in \Pi_{\psi_0}$.   Suppose that $\Pi_{0}= u_{\rho}(a,b) \rtimes \pi_{0}$ is reducible. Then there exists $\pi_1 , \pi_2 \leq \Pi_0$ such that 
\[ \langle s_u, \pi_{1}\rangle_{\psi}= - \langle s_u, \pi_{2}\rangle_{\psi}. \]
\end{thm}

For the rest of this subsection, we assume that Theorem \ref{thm base case reducible components} holds. As a consequence, we are able to obtain the non-unitarity of a certain parabolic induction.

\begin{cor}\label{cor base case non-unitary}
Let $\pi_{0}\in \Pi_{\psi_0}$.   Suppose that $\Pi_{0}= u_{\rho}(a,b) \rtimes \pi_{0}$ is reducible. Then, the irreducible representations $\Pi_{s}:= u_{\rho}(a,b)\lvert \cdot \rvert^{s} \rtimes \pi_{0}$ are not unitary for any $0 < s <\half{1}$.
\end{cor}

\begin{proof}
By Theorem \ref{thm base case reducible components}, there exists $\pi_1 , \pi_2 \leq \Pi_0$ such that 
\[ \langle s_u, \pi_{1}\rangle_{\psi}= - \langle s_u, \pi_{2}\rangle_{\psi}. \] Thus, $N(0)$ cannot be a scalar as \eqref{eq A-LIR} implies that $c(\pi_1)=-c(\pi_2).$ From Lemma \ref{lemma non-unitary}, it follows that $\Pi_s$ is non-unitary for $s>0$ close to 0. Since $\Pi_s$ is irreducible for $0<s<\frac{1}{2}$ by Theorem \ref{thm red from nu to gp}, it follows that $\Pi_s$ is not unitary for $0<s<\frac{1}{2}$. 
\end{proof}

This corollary will crucially serve as the base case in proving Theorem \ref{thm A+,u} later. Our next goal is to establish Theorem \ref{thm base case reducible components}.

\section{\texorpdfstring{Extended $\Z$-segments}{}}\label{sec Z-extended segments}

In this section, we define a notion which we call \emph{extended $\mathbb{Z}$-segments} and study certain non-vanishing criteria of them. The non-vanishing criteria are used to detect the non-vanishing of certain representations in local Arthur packets in Theorem \ref{thm non-vanishing}. The results in this section serve as the main technical tools for proving Theorem \ref{thm base case reducible components} and Theorem \ref{thm A+,u}.
\subsection{Basic Definitions}\label{subsec:extended-segments}

In this section, we introduce and study \emph{extended $\mathbb{Z}$-segments}, which will serve as fundamental combinatorial objects in our analysis.

\begin{defn}\label{def:extended-segments}
\begin{enumerate}
    \item A \emph{$\mathbb{Z}$-segment} is a finite set of consecutive integers
    \[
    \Delta = [A,B] := \{A, A-1, \ldots, B\},
    \]
    where $A, B \in \mathbb{Z}$ satisfy $A \geq B$. We call $b(\Delta) := \#\Delta = A - B + 1$ the \emph{length} of $\Delta$.

    \item A \emph{virtual extended $\mathbb{Z}$-segment} is a triple $\mathfrak{e} = ([A,B], l, \eta)$, where:
    \begin{itemize}
        \item $[A,B]$ is a $\mathbb{Z}$-segment,
        \item $l \in \mathbb{Z}$ satisfies $l \leq \frac{b}{2}$, here $b$ is the length of $[A,B]$,
        \item $\eta \in \{\pm 1\}/E$, where the equivalence relation $E$ is given by:
        \[
        E = \begin{cases}
            \{\pm 1\} & \text{if } b = 2l, \\
            \{+1\} & \text{if } b > 2l.
        \end{cases}
        \]
    \end{itemize}
    We call $\mathfrak{e}$ an \emph{extended $\mathbb{Z}$-segment} if additionally $l \geq 0$. The \emph{support} of $\mathfrak{e}$ is $\Supp(\mathfrak{e}) := [A,B]$.

    \item Denote by:
    \begin{itemize}
        \item $\Eseg$ the set of all extended $\mathbb{Z}$-segments,
        \item $\VEseg$ the set of all virtual extended $\mathbb{Z}$-segments.
    \end{itemize}
    For any $\mathbb{Z}$-segment $\Delta$, define the fiber sets:
    \[
    \Eseg_\Delta := \{\mathfrak{e} \in \Eseg \mid \Supp(\mathfrak{e}) = \Delta\}, \quad
    \VEseg_\Delta := \{\mathfrak{e} \in \VEseg \mid \Supp(\mathfrak{e}) = \Delta\}.
    \]
\end{enumerate}
\end{defn}

We represent an extended $\mathbb{Z}$-segment $\mathfrak{e} = ([A,B], l, \eta)$ diagrammatically as follows:
\[
\mathfrak{e} = 
\left(
\begin{array}{rcl}
\underbrace{\overset{B}{\lhd} \lhd \cdots \overset{B+l-1}{\lhd}}_{l} 
&
\overset{B+l}{\odot} \odot \cdots \odot \overset{A-l}{\odot} 
&
\underbrace{\overset{A-l+1}{\rhd} \cdots \rhd \overset{A}{\rhd}}_{l}
\end{array}
\right),
\]
where:
\begin{itemize}
    \item The middle section $\odot\cdots\odot$ represents an alternating sequence determined by $\eta$ as follows:
    \begin{itemize}
        \item it starts with $\oplus$ when $\eta = +1$;
        \item it starts with $\ominus$ when $\eta = -1$.
    \end{itemize}
    \item The numbers above the symbols indicate their integer positions.
\end{itemize}

\begin{remark}\ 
    \begin{enumerate}
        \item We only visualize \emph{extended} $\mathbb{Z}$-segments (where $l \geq 0$), not general virtual extended $\mathbb{Z}$-segments.
        \item Virtual extended $\mathbb{Z}$-segments are used to simplify certain definitions and ensure certain operators remain well-defined in all cases.
    \end{enumerate}
\end{remark}

Next, we define the concept of adjacent and interval for virtual extended segments. 

\begin{defn}\label{def adjacent}\ 
    \begin{enumerate}
    \item [(a)]  We say two virtual extended $\Z$-segments $\mfr{e}_{i}= ([A_i,B_i],l_i, \eta_i)$, $i=1,2$, are adjacent if $\Supp(\mfr{e}_1)=\Supp(\mfr{e}_2)$ and one of the following conditions holds.
        \begin{itemize}
            \item There exists a lift $\eta_i\in \{\pm 1\}$ such that $\eta_1=\eta_2$ and $|l_1-l_2|=1$.
            \item $A_1-B_1+1$ is odd, $l_1=l_2=\half{A_1-B_1}$, and $\eta_1=-\eta_2$.
        \end{itemize}
        \item [(b)] We say that a finite set $S$ of virtual extended $\Z$-segments is a virtual interval if $S=\{\ee_1,\ldots, \ee_r\}$, where $\ee_i$ and $\ee_{i+1}$ are adjacent (and $\ee_i \neq \ee_j$ for $i \neq j$). We define the support of $S$ to be $\Supp(S):= \Supp(\ee)$ for any $\ee \in S$ if $S$ is non-empty.  We say a virtual interval $S$ is an interval if $S \subseteq \Eseg$.        
        \item [(c)] We say two non-empty virtual intervals $S_1, S_2$ are adjacent if $S_1\cup S_2$ is a virtual interval and $|S_1|+1=|S_1\cup S_2|= |S_2|+1.$
        \item [(d)] We say two intervals $S_1, S_2$  in $\Eseg$ are adjacent if there exist virtual intervals $S_1^+, S_2^+$ such that $S_i= S_i^+ \cap \Eseg$ and $S_1^+$ are adjacent to $S_2^+$.
    \end{enumerate}
\end{defn}

\begin{remark}
    We explain a picture for the above definitions. For a fixed $\mathbb{Z}$-segment $\Delta = [A,B]$, the set $\Eseg_\Delta$ of extended $\mathbb{Z}$-segments has cardinality $|\Eseg_\Delta| = A - B + 2$. We define a total order $>$ on $\Eseg_\Delta$ as follows:

\begin{enumerate}
    \item When $\Delta$ has odd length $b = A - B + 1 = 2x + 1$:
    \[
    \begin{aligned}
    (\Delta, 0, +1) > (\Delta, 1, +1) > \cdots > (\Delta, x, +1) > \\
    (\Delta, x, -1) > (\Delta, x-1, -1) > \cdots > (\Delta, 0, -1).
    \end{aligned}
    \]
    
    \item When $\Delta$ has even length $b = 2x$:
    \[
    \begin{aligned}
    (\Delta, 0, +1) > (\Delta, 1, +1) > \cdots > (\Delta, x, +1) = \\
    (\Delta, x, -1) > (\Delta, x-1, -1) > \cdots > (\Delta, 0, -1).
    \end{aligned}
    \]
\end{enumerate}

We can fix an order-preserving embedding $\iota: \Eseg_{\Delta} \cong \{  1, \ldots, A-B+2\}\subseteq \R$. Then, $\ee_1, \ee_2 \in \Eseg_{\Delta}$ are \emph{adjacent} if $|\iota(\ee_1)-\iota(\ee_2)|=1$, and $S \subseteq \Eseg_{\Delta}$ is an \emph{interval} if $\iota(S)= (x,y) \cap \iota(\Eseg_{\Delta})$ for some open interval $ (x,y) =\{z \in \R \ | \ x > z > y\} \subseteq \R$. Two intervals $S_1, S_2 \subseteq \Eseg_{\Delta}$ are \emph{adjacent} if, after relabeling if necessary, there exists an interval $I_1= (x,y)$ and $I_2=(x+1,y+1)$ such that $\iota(S_i)= I_i \cap \iota(\Eseg_{\Delta})$. 

The above interpretation of the notation has the drawback that the choice of the total order is not canonical. That is why we give the relatively abstract definition in Definition \ref{def adjacent}. However, the geometric picture is helpful to understand the motivation of the proof below.
\end{remark}

\begin{exmp}
    Consider $\Delta=[1,0]$. Then $\Eseg_{\Delta}= \{ \ee_1,\ee_2,\ee_3 \},$ where
    \begin{align*}
        \ee_1= \bordermatrix{ & 0&1 \cr  
        & \ominus & \oplus },\ \ 
        \ee_2= \bordermatrix{ & 0&1 \cr  
        & \lhd & \rhd },\ \ 
         \ee_3= \bordermatrix{ & 0&1 \cr  
      & \oplus & \ominus }.
    \end{align*}
    We also consider the following virtual extended segments 
\[\ee_{-1}:= ([1,0],-2,-1),\  \ee_{0}:= ([1,0],-1,-1),\ \  \ee_{4}=([1,0],-1,1),\ \ee_{5}=([1,0],-2,1). \]
Then the extended segment $\ee_i$ is adjacent to the (virtual) extended segments $\ee_{i+1}$ and $\ee_{i-1}$.    
Among the $8$ subsets of $\Eseg_{\Delta}$, only $\{\ee_1, \ee_3\}$ is not an interval. 

Next, we list all intervals that are adjacent to $S:=\{\ee_1\}$. Indeed, we may regard $\{\ee_1\}$ as the intersection of $\Eseg$ with the virtual intervals $S^{+}_1=\{\ee_1\}$ or $S^{+}_2=\{\ee_0, \ee_1\}$. The virtual interval $S^+_1$ is adjacent to the two virtual intervals $\{\ee_0\}$ and $\{\ee_2\}$, and the virtual interval $S^+_2$ is adjacent to the two virtual intervals $\{\ee_{-1},\ee_{0}\}$ and $\{\ee_{1},\ee_2\}$. By intersecting these virtual intervals with $\Eseg_{\Delta}$, we conclude that the interval $S$ is adjacent to the following three intervals
\[ \emptyset,\  \{\ee_2\},\  \{\ee_{1},\ee_2\}. \]
It is not hard to see that if we replace $S_{2}^+$ by any virtual interval $(S_2^{+})'$ such that $(S_2^{+})' \supset S_2^{+}$ and $(S_2^{+})'\cap \Eseg= S$, the result is unchanged. 

By a similar discussion, one can see that the interval $\{\ee_1,\ee_2,\ee_3\}$ is adjacent to $ \{\ee_1,\ee_2,\ee_3\}$, $\{\ee_1,\ee_2\},$ and $\{\ee_2,\ee_3\}.$
Also, the interval $\emptyset$ is adjacent to $\emptyset, \{\ee_1\}$ and $\{\ee_3\}$.
\end{exmp}

We list several observations, which are crucial for the proof of Proposition \ref{prop NV segment} below, the key reduction statement of our main result.
\begin{lemma}\label{lem key observation}
Let $ S_1, S_2$ be two non-empty intervals with $\Supp(S_1)=\Supp(S_2)$.
\begin{enumerate}
    \item [(a)] The set $S_1 \cap S_2$ is an interval.
    \item [(b)] There are at most $3$ intervals that are adjacent to $S_1$.
    \item [(c)] Suppose that the following conditions hold.
    \begin{itemize}
        \item $|S_1|>1 $, $|S_2|>1$ and $|S_1 \cap S_2 |=1$.
        \item There exists an interval $S_1'$ adjacent to $S_1$ and $|S_1' \cap S_2 | >0$.
    \end{itemize}
    Then $|S_1' \cap S_2 | =2$.
\end{enumerate}
\end{lemma}
\begin{proof}
    These are straightforward consequences of the definitions.
\end{proof}

\subsection{\texorpdfstring{Non-vanishing criterion for ordered pairs of extended $\Z$-segments}{}}

Given two extended $\Z$-segments, we define when they form a non-vanishing pair motivated by \cite[Lemmas 5.5, 5.6, 5.7]{Xu21b}. This notation will be used to detect the non-vanishing of certain representations in local Arthur packets in Theorem \ref{thm non-vanishing}. 

\begin{defn}\label{def non-vanishing}
Let $(\ee_1, \ldots, \ee_n)$ be a sequence of virtual extended $\Z$-segments and write $\mfr{e}_{i}= ([A_i,B_i],l_i, \eta_i)$.
\begin{enumerate}
    \item   We say the sequence $(\ee_1, \ldots, \ee_n)$ or $(\Supp(\ee_1),\ldots, \Supp(\ee_n))$ is admissible if
\begin{align*}
    \tag{P}  A_i< A_j, B_i< B_j\Longrightarrow i<j.  
\end{align*}
    \item Suppose $(\ee_1, \ee_2)$ is admissible.  Set $b_i:=A_i-B_i+1$ and $\epsilon:= (-1)^{A_1-B_1} \eta_1 \eta_2$ by taking any lift $\eta_i \in \{\pm 1\}$. We say that the ordered pair $(\mfr{e}_1, \mfr{e}_2)$ satisfies the non-vanishing criterion, and write $\NV(\ee_1, \ee_2) \neq 0$, if $\ee_1, \ee_2 \in \Eseg$ and the following conditions hold.
    \begin{enumerate}
    \item [(a)] If $A_1 \leq A_2$ and $B_1 \leq B_2$, then
    \[ \begin{cases}
        \epsilon=1 &\Rightarrow\,\,\, B_1+l_1 \leq B_{2}+l_{2}, \ A_1-l_1 \leq A_{2}-l_{2};\\
        \epsilon=-1 &\Rightarrow\,\,\,  A_1-l_1 < B_{2}+l_{2}.
        \end{cases} \]

         \item [(b)] If $A_1 \leq A_2$ and $B_1 \geq B_2$ $(i.e., [A_1,B_1] \subseteq [A_2,B_2])$, then
        \[ \begin{cases}
        \epsilon=1 &\Rightarrow\,\,\, 0 \leq l_{2} -l_{1} \leq b_{2}-b_{1};\\
        \epsilon=-1 &\Rightarrow \,\,\, l_1+l_{2} \geq b_{1}.   \end{cases} \]

        \item  [(c)] If $ A_1 \geq A_2$ and $B_1 \leq B_2$ $(i.e., [A_1,B_1]\supseteq [A_2,B_2])$, then
        \[ \begin{cases}
        \epsilon=1 &\Rightarrow\,\,\, 0 \leq l_{1} -l_{2} \leq b_{1}-b_{2};\\
        \epsilon=-1 &\Rightarrow\,\,\,  l_1+l_{2} \geq b_{2}.    \end{cases} \]

\end{enumerate}
Otherwise, we write $\NV(\ee_1,\ee_2)=0$.
\item Suppose $(\ee_1,\ee_2)$ is admissible and $\Supp(\ee_i)= [A_i, B_i]$. We write $\ee_1 \preceq \ee_2$ or $\Supp(\ee_1) \preceq \Supp(\ee_2)$ if 
\begin{itemize}
    \item  $[A_1,B_1] \subseteq [A_2,B_2]$; or, 
    \item $[A_1,B_1] \not\supseteq [A_2,B_2]$, $[A_1,B_1] \not\subseteq [A_2,B_2]$, and $B_1<B_2.$
\end{itemize}
\end{enumerate}
\end{defn}

\begin{remark}\label{rmk non-vanishing}\
\begin{enumerate}
    \item One can check that the above Conditions (a), (b) and (c) are independent of the choices of the lifting of $\eta_i$ in $\{\pm 1\}$.
    \item If $A_1=A_2$, then Conditions (a) and (b) are equivalent. If $[A_1,B_1]=[A_2,B_2]$, then (a), (b) and (c) are all equivalent. Moreover, in this case $\NV(\ee_1, \ee_2) \neq 0$ if and only if $l_1=l_2$ and there is a lifting of $\eta_i \in \{\pm 1\}$ such that $\epsilon=1$. Thus for an extended $\Z$-segment $\ee$, we shall denote $\ee^{\dagger}$ the unique extended $\Z$-segment such that $\Supp(\ee)= \Supp(\ee^{\dagger}    )$ and    
    $\NV(\ee, \ee^{\dagger})\neq 0$.
\end{enumerate}
\end{remark}

Fixing an extended $\Z$-segment $\ee$, we define the set of extended $\Z$-segments which form a non-vanishing pair with $\ee$.

\begin{defn}
    Let $\ee=([A,B],l,\eta)$ be an extended $\Z$-segment and $\Delta'=[A',B']$ be a $\Z$-segment. We define
    \begin{align*}
        \NV_{(\ee,-)}(\Delta')&:= \{ \ee' \in \Eseg_{\Delta'} \ | \ \NV(\ee,\ee') \neq 0  \},\\
        \NV_{(-,\ee)}(\Delta')&:= \{ \ee' \in \Eseg_{\Delta'} \ | \ \NV(\ee',\ee) \neq 0  \}.
    \end{align*}
\end{defn}
Note that the set $\NV_{(\ee,-)}(\Delta')$ is empty if the sequence $(\Supp(\ee), \Delta')$ is not admissible. Conversely, if $(\Supp(\ee), \Delta')$ is admissible, then $\NV_{(\ee,-)}(\Delta')$ is non-empty. We list properties of the set $\NV_{(\ee,-)}(\Delta')$ in the following lemma.

\begin{lemma}\label{lem adj} Let $\ee_1 ,\ee_2 \in \Eseg$ such that $(\ee_1,\ee_2)$ is admissible. Write $\Supp(\ee_i)=\Delta_i$.
\begin{enumerate}
    \item [(i)] The sets $\NV_{(\ee_1,-)}(\Delta_2)$ and $\NV_{(-,\ee_2)}(\Delta_1)$ are non-empty intervals.
    \item [(ii)] If $\ee_1$ is adjacent to $\ee_1'$, then $\NV_{(\ee_1,-)}(\Delta_2)$ is adjacent to $\NV_{(\ee_1',-)}(\Delta_2)$. Similarly, if $\ee_2$ is adjacent to $\ee_2'$, then $\NV_{(-,\ee_2)}(\Delta_1)$ is adjacent to $\NV_{(-,\ee_2')}(\Delta_1)$.
    \item [(iii)] Suppose $\NV(\ee_1, \ee_2) \neq 0$ and $\ee_1 \preceq \ee_2$. Then $|\NV_{(\ee_1, -)}(\Delta_2)|=1$ if and only if $\Delta_1 = \Delta_2$. Similarly, suppose $\NV(\ee_1, \ee_2) \neq 0$ and $\ee_2 \preceq \ee_1$. Then $|\NV_{( -,\ee_2)}(\Delta_1)|=1$ if and only if $\Delta_1 = \Delta_2$. 
\end{enumerate}
\end{lemma}
\begin{proof}
We are going to write down the set $\NV_{(\mfr{e}_1,-)}(\Delta_2)$ explicitly. Then we verify half of Part (i) that $\NV_{(\ee_1,-)}(\Delta_2)$ is an interval. Parts (ii) and (iii) will also be a direct consequence of the computation. We omit the analogous verification of these statements for $\NV_{(-,\ee_2)}(\Delta_1)$.

 Write $\ee_1= ([A_1,B_1], l_1, \eta_1)$ and $\Delta_2=[A_2, B_2]$. Note that we have assumed $A_1,A_2 \in  \Z$ in the definition of extended $\Z$-segments. We consider the three cases that satisfy Conditions (a), (b) or (c) in Definition \ref{def non-vanishing}(2) respectively.

\textbf{Case (a):}  $A_1 \leq A_2$ and $B_1\leq B_2$.

In this case, $ \NV_{(\mfr{e_1},-)}(\Delta_2)= S_1 \cup S_2$, where
\begin{align*}
   S_1&= \{ ([A_2,B_2], l_2, (-1)^{A_1-B_1}\eta_1) \in \Eseg\ | \ B_1-B_2+l_1 \leq l_2 \leq A_2-A_1+l_1 \},\\
   S_2&= \{ ([A_2,B_2], l_2, (-1)^{A_1-B_1+1}\eta_1) \in \Eseg \ | \ A_1-B_2-l_1 < l_2\}.
\end{align*}
Recall that there is a natural restriction that $0 \leq l_2\leq \half{A_2-B_2+1}$ in order that the triple is an extended $\Z$-segment. It is clear that both $S_1$ and $S_2$ are intervals. To show that the union is also an interval, we assume both $S_1$ and $S_2$ are non-empty. Since $S_2$ is non-empty, there exists an $l_2$ such that 
\[ A_1-B_2-l_1 < l_2 \leq \half{A_2-B_2+1}. \]
 Therefore, 
 \begin{align}\label{eq Case (a) 0}
     A_2- A_1 +l_1  > \half{A_2-B_2+1}-1. 
 \end{align}
Thus, if $b_2:= A_2-B_2+1= 2x+1$ is odd, then (we have assumed that $S_1$ is non-empty) 
 \[ (\Delta_2, x, (-1)^{A_1-B_1}\eta_1) \in S_1,\ \  (\Delta_2, x, (-1)^{A_1-B_1+1}\eta_1) \in S_2.  \]
If $b'= 2x$ is even, then 
 \[ (\Delta_2, x-1, (-1)^{A_1-B_1}\eta_1) \in S_1,\ \ (\Delta_2, x, (-1)^{A_1-B_1}\eta_1)= (\Delta_2, x, (-1)^{A_1-B_1+1}\eta_1) \in S_2.  \]
 In any case, this implies that $S_1 \cup S_2$ is an interval. Also, from the explicit description of $S_1$ and $S_2$,  a case-by-case straightforward computation implies Part (ii), which we omit.

Finally, we check Part (iii). Note that in this case, we have $\ee_1 \preceq \ee_2$ unless $ A_1=A_2$. We leave the case that $A_1=A_2$ to Case (c). Thus, in the following discussion we assume that $A_1 < A_2$, and we are going to verify that $|\NV_{(\ee_1,-)}(\Delta_2)|=1$ if and only if $\Delta_1=\Delta_2$. One direction is already commented in Remark \ref{rmk non-vanishing}(2). We also verify that $\NV_{(\ee_1,-)}(\Delta_2)$ is non-empty along the way.

{
If $|S_2|=0$, then 
\[A_1- B_2-l_1 \geq \half{A_2-B_2+1},\]
which implies that 
\begin{align*}
    A_2- A_1 + l_1 \leq  \half{A_2-B_2+1}-1.
\end{align*}
On the other hand, since $B_1 -B_2 \leq 0$, $2l_1 \leq A_1-B_1+1 $ and $A_1 \leq A_2$, we obtain
\[(B_1-B_2)+ B_1 +2 l_1 \leq B_1 +2 l_1 \leq A_1+1 \leq A_2+1,\]
which implies that 
\[B_1-B_2+ l_1 \leq \half{A_2 -B_2+1}. \]
Therefore, we conclude that $S_1$ is non-empty if $S_2$ is empty. Suppose further that $|S_1|=1$ (and $|S_2|=0$). Then one of the following situations holds. Here we write $A_2-B_2+1=2x$ or $A_2-B_2+1=2x+1$.
\begin{itemize}
     \item [(1)] $B_1-B_2+l_1 < 0$ and $A_2-A_1+l_1=0.$
   
    \item [(2)] $0 \leq B_1-B_2+l_1= A_2-A_1+l_1 \leq x$.
    
     \item [(3)] $B_1-B_2+l_1 = x$ and $A_2-A_1+l_1>x$.
\end{itemize}
In situation (1), we have $A_1=A_2$ and $l_1=0$, which is the case we defer to Case (c). In situation (2), we have $A_1+B_1=A_2+B_2$, which implies that $\Delta_1=\Delta_2$ since $A_1 \leq A_2$ and $B_1\leq B_2$. Finally, situation (3) is indeed empty since the conditions that $|S_2|=0$ and $A_2-A_1+l_1>x$ contradict \eqref{eq Case (a) 0}.

Next, we consider the case that $|S_2|=1$ and $|S_1|=0$. Note that since  $A_1 \leq A_2$, $B_1\leq B_2$ and $l_1 \geq 0$, we have 
\[\begin{cases}
    B_1 - B_2 + l_1 \leq A_2-A_1+l_1, \text{ and }\\
    A_2-A_1+l_1 \geq 0.
\end{cases}  \]
Thus, there always exists an $l_2 \in \Z_{\geq 0}$ such that $ B_1-B_2+l_1 \leq l_2 \leq A_2-A_1+l_1$ holds. As a consequence, the set $S_1$ is empty only if this $l_2$ violates the natural restriction $l_2 \leq \half{A_2-B_2+1}$. In other words, if $|S_1|=0$, then
\begin{align}\label{eq Case (a) 1}
    B_1-B_2+l_1> \half{A_2-B_2+1}.
\end{align}
On the other hand, $|S_2|=1$ only if 
\begin{align}\label{eq Case (a) 2}
    A_1-B_2-l_1= \lfloor \half{A_2-B_2+1} \rfloor -1.
\end{align}
Putting \eqref{eq Case (a) 1} and \eqref{eq Case (a) 2} together, we obtain that
\[A_1-B_1+1 -2l_1= (A_1-B_2-l_1)- (B_1-B_2+l_1)+1 <\lfloor \half{A_2-B_2+1} \rfloor 
 -  \half{A_2-B_2+1} \leq 0. \]
Since $l_1 \leq \half{A_1-B_1+1}$, the above inequality implies that $A_2-B_2+1$ is even and $l_1=\half{A_1-B_1+1}$. Then \eqref{eq Case (a) 2} becomes
\[ A_1- B_2- \half{A_1-B_1+1}= \half{A_2-B_2+1} -1. \] 
Hence, $ A_1+B_1= A_2+B_2$. The assumption that $A_1 \leq A_2$, $B_1\leq B_2$ again implies that $\Delta_1=\Delta_2$.

Finally, we consider the case that $|S_1|=|S_1\cup S_2|= |S_2|=1$. In this case, it is necessary that $A_2-B_2+1 =2x$ is even and 
\[ S_1= \{([A_2,B_2], x, (-1)^{A_1-B_1}\eta_1)\}=\{([A_2,B_2], x,  (-1)^{A_1-B_1+1}\eta_1) \}= S_2. \]
By definition, we obtain that
\[\begin{cases}
    B_1-B_2+l_1= x,\\
    A_1-B_2-l_1= x-1.
\end{cases}\]
Therefore,
\begin{align*}
    A_1+B_1= B_1+B_2+l_1+x-1= B_1+B_2+ (x-B_1+B_2) +x-1= 2B_2+2x-1=A_2+B_2.
\end{align*}
Since $A_1 \leq A_2$ and $B_1 \leq B_2$, we conclude that $ \Delta_1= \Delta_2$. This completes the verification of Part (iii) in this case.
}

\textbf{Case (b):} $A_1 \leq A_2$ and $B_1\geq B_2$, i.e. $[A_1,B_1]\subseteq [A_2, B_2]$.

In this case, $ \NV_{(\mfr{e}_1,-)}(\Delta)= S_1 \cup S_2$, where
\begin{align*}
   S_1&= \{ ([A_2,B_2], l_2, (-1)^{A_1-B_1}\eta_1) \in \Eseg\ | \  l_1 \leq l_2 \leq b_2- b_1+ l_1 \},\\
   S_2&= \{ ([A_2,B_2], l_2, (-1)^{A_1-B_1+1}\eta_1) \in \Eseg \ | \ b_1-l_1 \leq l_2 \}.
\end{align*}
Again, $S_1, S_2$ are both intervals if non-empty, and we show that $S_1 \cup S_2$ is an interval assuming that $S_1$ and $S_2$ are both non-empty. Since $S_1$ is non-empty, we have $b_1 -l_1 \leq \half{b_2}$, which implies that $    b_2-b_1+ l_1 \geq \half{b_2}.$
Thus, the same argument in Case (a) implies that $S_1 \cup S_2$ is an interval. This proves Part (i) in this case. Part (ii) follows from a case-by-case straightforward computation, which we also omit.

We explain Part (iii) now. In this case, we have $\ee_1 \preceq \ee_2$, and we are going to show that $|S_1 \cup S_2|=1$ if and only if $\Delta_1=\Delta_2$. First, the assumption that $[A_1, B_1] \subseteq [A_2,B_2]$ implies that $b_1 \leq b_2$. Thus, we have
\begin{align}\label{eq lem adj case(b)}
l_1 \leq \half{b_1} \leq \half{b_2},
\end{align}
and $S_1$ always contains $([A_2,B_2], l_1, (-1)^{A_1-B_1}\eta_1)$. In order that $S_1$ itself is a singleton, we have either 
\begin{enumerate}
    \item [(i)]$l_1= b_2-b_1+l_1$, or
    \item [(ii)]$l_1 < b_2-b_1+l_1$ and $l_1= \lfloor \half{b_2} \rfloor$.
\end{enumerate}
In situation (i), we obtain $b_1=b_2$, which implies that $\Delta_1 \subseteq \Delta_2$ is indeed an equality. In situation (ii), \eqref{eq lem adj case(b)} implies that we must have $b_1= 2l_1$ and $b_2=b_1+1$ is odd. Then $S_2$ contains $([A_2,B_2], l_1, (-1)^{A_1-B_1+1}\eta_1)$ and $|S_1 \cup S_2|\geq 2$. This completes the verification of Part (iii) in this case.

\textbf{Case (c):} $A_1 \geq A_2$ and $B_1\leq B_2$, i.e. $[A_1,B_1]\supseteq [A_2, B_2]$. 

In this case, $ \NV_{(\mfr{e}_1,-)}(\Delta)= S_1 \cup S_2$, where
\begin{align*}
   S_1&= \{ ([A_2,B_2], l_2, (-1)^{A_1-B_1}\eta_1) \in \Eseg\ | \   b_2- b_1+l_1 \leq l_2 \leq l_1 \},\\
   S_2&= \{ ([A_2,B_2], l_2, (-1)^{A_1-B_1+1}\eta_1) \in \Eseg \ | \ b_2 -l_1 \leq l_2\}.
\end{align*}
We omit the rest of the verification, which is similar to Case (b). This completes the proof of the lemma.
\end{proof}

\subsection{Row exchange operator}

In this subsection, we define a row exchange operator on ordered pairs of extended $\Z$-segments. This is related to the row exchange operator on extended multi-segments which parameterize representations in local Arthur packets (\cite[\S6]{Xu21b}).

\begin{defn}  \cite[Section 4.2]{Ato20b} \label{def row exchange} 
Suppose that $\ee_i= ([A_i,B_i],l_i,\eta_i) \in \VEseg$, $i=1,2$, satisfies that $[A_1,B_1] \supseteq [A_2,B_2]$ or $[A_1,B_1] \subseteq [A_2,B_2]$. We define
\[ R(\ee_1,\ee_2):= (\ee_2', \ee_1'),\]
where $\ee_i'= ([A_i,B_i], l_i',\eta_i')$ is given as follows: Set $b_i:= A_i-B_i+1$ and $\epsilon:= (-1)^{A_1-B_1} \eta_1 \eta_2$ for any fixed lifting $\eta_i \in \{\pm 1\}$.
\begin{enumerate}
  
    \item [\textbf{Case 1.}] $ [A_1,B_1] \subseteq [A_{2},B_{2}]$:
        In this case, we set $(l_{1}',\eta_{1}')=(l_{1}, (-1)^{A_{2}-B_{2}}\eta_{1})$, and
    \begin{enumerate}
   \item [(a)] If $\epsilon=1$ and $b_{2}- 2l_{2} < 2(b_{1}-2l_{1})$, then
    \[ (l_{2}', \eta_{2}')= (b_{2}-(l_{2}+ (b_{1}-2l_{1})), (-1)^{A_{1}-B_{1}} \eta_{2}).  \]
    \item [(b)] If $\epsilon=1$ and $b_{2}- 2l_{2} \geq  2(b_{1}-2l_{1})$,
    then
    \[ (l_{2}', \eta_{2}')= (l_{2}+ (b_{1}-2l_{1}), (-1)^{A_{1}-B_{1}+1} \eta_{2}).  \]
    \item [(c)] If $\epsilon=-1$, then
    \[ (l_{2}', \eta_{2}')= (l_{2}- (b_{1}-2l_{1}), (-1)^{A_{1}-B_{1}+1} \eta_{2}).  \]
\end{enumerate}
  \item [\textbf{Case 2.}] $ [A_1,B_1] \supsetneq [A_{2},B_{2}]$:
        In this case, we set $(l_{2}',\eta_{2}')=(l_{2}, (-1)^{A_1-B_1}\eta_{2})$, and
    \begin{enumerate}
    \item [(a)] If $\epsilon=1$ and $b_1- 2l_1 < 2(b_{2}-2l_{2})$, then
    \[ (l_1', \eta_{1}')= (b_1-(l_1+ (b_{2}-2l_{2})), (-1)^{A_{2}-B_{2}} \eta_1).  \]
    \item [(b)] If $\epsilon=1$ and $b_1- 2l_1 \geq  2(b_{2}-2l_{2})$, then
    \[ (l_{1}', \eta_{1}')= (l_1+ (b_{2}-2l_{2}), (-1)^{A_{2}-B_{2}+1} \eta_1).  \]
    \item [(c)] If $\epsilon=-1$, then
    \[ (l_{1}', \eta_{1}')= (l_1- (b_{2}-2l_{2}), (-1)^{A_{2}-B_{2}+1} \eta_1).  \]
\end{enumerate}
\end{enumerate}
\end{defn}

\begin{remark}\label{rmk row exchange}
Suppose that $R(\ee_1,\ee_2)= (\ee_2', \ee_1')$.
\begin{enumerate}
    \item One can check that $\ee_2'$ and $\ee_1'$ are still virtual extended $\Z$-segments. Namely, if $\ee_i'= ([A_i,B_i],l_i',\eta_i')$, then $l_i' \leq \half{b_i} $.
    \item Suppose that $\ee_1, \ee_2 \in \Eseg$ and
    $\NV(\ee_1, \ee_2) \neq 0$. Then the inequalities in Parts (b) and (c) of Definition \ref{def non-vanishing}(2) guarantee that $\ee_1', \ee_2'$ are still extended $\Z$-segments. For example, in Case 1 and $\epsilon=-1$, the non-vanishing criterion gives $l_1 +l_2 \geq b_1$. Hence, 
\[ l_2 -(b_1-2l_1)= l_1+ (l_1+l_2 -b_1) \geq l_1 \geq 0. \]
The other cases can be verified similarly and we omit the details.
\item The following facts can be verified by direct computations:
\begin{itemize}
    \item If $\NV(\ee_1,\ee_2) \neq 0,$ then $\NV(\ee_2', \ee_1' ) \neq 0$.
    \item $R(\ee_2',\ee_1')= (\ee_1,\ee_2)$.
\end{itemize}
\end{enumerate}
\end{remark}

We illustrate the special case that $l_2=0$ in Case 2. Similar discussion applies to Case 1 when $l_1=0$.
\begin{exmp}\label{exmp l=0 row exchange}
    Consider $\ee_1=([A_1,B_1],l_1,\eta_1)$ and $\ee_2^k= ([B_1+k, B_1],0, \eta_2)$ for $1 \leq k \leq K_{\epsilon}$ where
    \[ K_{\epsilon}= \begin{cases}
        b_2-l_2 & \text{ if }\epsilon=(-1)^{A_1-B_1}\eta_1\eta_2=1,\\
        l_1 & \text{ if }\epsilon=-1.
    \end{cases}\]
     Then $\NV(\ee_1,\ee_2^k)\neq 0$. Write $R(\ee_1,\ee_2^k)= ((\ee_2^k)^{\ast}, \ee_1^k)$. 
    
    When $k=1$, we have  $\ee_1^1= ([A_1,B_1],l_1^1,\eta_1^1)$
    \[(l_1^1,\eta_1^1)=\begin{cases}
        (l_1+1, -\eta_1) & \text{ if }\varepsilon=1 \text{ and }b_1-2l_1> 1, \\
        (b_1-l_1, \eta_1) & \text{ if } \varepsilon=1 \text{ and }b_1-2l_1\leq  1,\\
        (l_1-1, -\eta_1)  & \text{ otherwise.}
    \end{cases}\]
   In other words, the resulting $\Z$-extended segment $\ee_1^1$ is adjacent to $\ee_1^0=([A_1,B_1],l_1, -\eta_1)$, and the sign $\epsilon$ determines the direction.

   For general $k$, a direct computation shows that $\{\ee_1^0, \ee_1^1, \ldots, \ee_1^k \}$ is an interval of length $k+1$. An end point of this interval is $\ee_1^0$, and $\epsilon$ determines it is the left or right end point. The upper bound $K_{\epsilon}$ above exactly reflects the maximal length of intervals with $\ee_1^0$ as the left or right end point.

   From the above observation, it is not hard to check that for $1 \leq k' < k \leq K_{\epsilon}$, we may break $\ee_2^k$ into two pieces: $\ee_2^{k'}$ and
   \[ (\ee_2^{k'})^c:= ([ B_2+k, B_2+k'+1],0, (-1)^{k'}\eta_2).\]
   Then write
   \[ R_{2} \circ R_1 ( \ee_1, \ee_2^{k'}, (\ee_2^{k'})^c )= R_2((\ee_{2}^{k'})^{\ast}, \ee_1^{k'}, (\ee_2^{k'})^c )=((\ee_{2}^{k'})^{\ast}, ((\ee_2^{k'})^c)^{\ast}, (\ee_1^k)'). \]
   Then we have $(\ee_1^k)'= \ee_1^k$.   
\end{exmp}

Here are some implications of the definition above. 

\begin{lemma}\label{lem observation row exchange} 
Suppose that $\ee_1, \ee_2 \in \VEseg$ and
$R(\ee_1, \ee_2)=  (\widetilde{\ee_2}, \widetilde{\ee_1})$ and $R  (\ee_1', \ee_2)=(\widetilde{\ee_2}', \widetilde{\ee_1}')$.
\begin{enumerate}
    \item If $\ee_1 \neq \ee_1'$, then $ \widetilde{\ee_1} \neq \widetilde{\ee_1}'$.
    \item If $\ee_1, \ee_1'$ are adjacent, then $\widetilde{\ee_1}$ and $\widetilde{\ee_1}'$ are adjacent.
\end{enumerate}
In other words, $R(-, \ee_2): \ee_1 \mapsto \widetilde{\ee_1}'$ is a bijection on $\VEseg_{\Supp(\ee_1)}$ sending virtual intervals to virtual intervals. The same statement holds for $R(\ee_1, -)$.
\end{lemma}
\begin{proof}
Part (1) is a direct consequence of the definition. We check Part (2).   If $\Supp(\ee_1) \subseteq \Supp(\ee_2)$, the conclusion is clear. Assume $\Supp(\ee_1) \supseteq \Supp(\ee_2)$ in the rest of the proof. Write $\ee_i= ([A_i,B_i],l_i, \eta_i )$ for $i=1,2$. 

    First, we consider the case that $ \ee_1'=([A_1,B_1],l_1+1, \eta_1 ) $ (relabeling $\ee_1$ and $\ee_1'$ if necessary). If $\epsilon:= (-1)^{A_1-B_1}\eta_1 \eta_2=-1$, or $\epsilon=1$ but the signs of $ b_1 -2 l_1 -(2 b_2 -2l_2)$ and $ b_1 -2 (l_1+1) -(2 b_2 -2l_2)$ are the same, then the desired conclusion follows from the formula in Case 2 in Definition \ref{def row exchange}. Thus we assume that 
    \[ \begin{cases}
        b_1-2l_1 \geq 2 (b_2 -2l_2),\\
        b_1-2l_1 -2 < 2 (b_2 -2l_2).
    \end{cases}\]
    If $b_1$ is even, then the first inequality must be an equality. Then
    \begin{align*}
        \widetilde{\ee_1}= ( [A_1,B_1], \half{b_1}, (-1)^{A_2-B_2+1}\eta_1),\ \         \widetilde{\ee_1}'= ( [A_1,B_1], \half{b_1}-1, (-1)^{A_2-B_2+1}\eta_1),
    \end{align*}
which are adjacent. If $b_1=2x+1$ is odd, then $b_1-2l_1= 2(b_2-2l_2)+1$. Hence, 
 \begin{align*}
        \widetilde{\ee_1}= ( [A_1,B_1], x, (-1)^{A_2-B_2+1}\eta_1),\ \         \widetilde{\ee_1}'= ( [A_1,B_1], x, (-1)^{A_2-B_2}\eta_1),
    \end{align*}
 which are again adjacent.

 Next, we deal with the case that $b_1=2x+1$ is odd and $\ee_1= ([A_1,B_1],x,\eta_1)$ and $\ee_1'= ([A_1,B_1],x,-\eta_1)$. Assume $\epsilon:=(-1)^{A_1-B_1}\eta_1\eta_2=-1$. If $b_2-2l_2=0 \leq 1= b_1-2l_1$, then
 \[ \widetilde{\ee_1}= ( [A_1,B_1], x, (-1)^{A_2-B_2+1}\eta_1), \ \ \widetilde{\ee_1}'=( [A_1,B_1], x, (-1)^{A_2-B_2}\eta_1),  \]
which are adjacent. If $ b_1 -2l_1=1 < b_2-2l_2, $ then
 \[ \widetilde{\ee_1}= ( [A_1,B_1], x -(b_2-2l_2), (-1)^{A_2-B_2+1}\eta_1), \ \ \widetilde{\ee_1}'=( [A_1,B_1], x+1-(b_2-2l_2), (-1)^{A_2-B_2+1}\eta_1),  \]
which are again adjacent. This completes the proof of the lemma.
\end{proof}
\subsection{\texorpdfstring{Non-vanishing criterion on admissible sequences of extended $\Z$-segments}{}}

In this subsection, we define a non-vanishing criterion on admissible sequences of extended $\Z$-segments.

\begin{defn}\label{def non-vanishing EEE}
    Let $\EEE=(\ee_1,\ldots, \ee_n)$ be an admissible sequence of virtual extended $\Z$-segments. 
    \begin{enumerate}
        \item [(1)] For $k=1,\ldots, n-1$, if $\Supp(\ee_k) \supseteq \Supp(\ee_{k+1})$ or $\Supp(\ee_{k+1}) \supseteq \Supp(\ee_{k})$, we define 
        \[ R_k(\EEE)= (\ee_{1},\ldots, \ee_{k-1}, \ee_{k+1}', \ee_{k}', \ee_{k+2},\ldots, \ee_n),\]
        where $R(\ee_{k},\ee_{k+1})= (\ee_{k+1}', \ee_{k}')$
        \item [(2)] We write $\EEE\stackrel{R}{=} \EEE'$ if 
        \[ \EEE' = R_{k_1} \circ \cdots \circ R_{k_s} (\EEE)\]
        for some $1 \leq k_i \leq n-1$. We let
        \[ [\EEE]:= \{\EEE' \ | \ \EEE' \stackrel{R}{=} \EEE\}.\]    
        \item [(3)] We write $ \widetilde{\NV}(\EEE) \neq 0$ if $\NV(\ee_{i}, \ee_{i+1})\neq 0$ for all $1 \leq i \leq n-1$ (see Definition \ref{def non-vanishing}).
        \item [(4)]  We say $\EEE$ satisfies the non-vanishing criterion, and write $\NV(\EEE) \neq 0$, if $\widetilde{\NV}(\EEE')\neq 0$ for all $ \EEE' \in [\EEE]$. Otherwise, we write $\NV(\EEE)=0$.
        \item [(5)] Write $\Supp(\ee_i)=[A_i, B_i]$. We say $\EEE$ satisfies $(P'')$ if 
\begin{align*}
    \tag{$P''$} i<j \Longrightarrow \begin{cases}
        B_i\leq B_j,\\
        A_i \geq A_j \text{ if }B_i=B_j.
    \end{cases}
\end{align*}
We define a non-empty subset $[\EEE]^{(P'')} \subseteq [\EEE]$ that consists of $\EEE'$ satisfying $(P'')$.
    \end{enumerate}
\end{defn}

 The following two statements will be proved in \S \ref{sec non-vanishing extended multi-segments}.
 
\begin{lemma}\label{lem P''}
If $\NV(\EEE)\neq 0$, then the set $[\EEE]^{(P'')}$ is a singleton. We denote the unique member by $\EEE^{(P'')}$.
\end{lemma}

By definition, $\NV(\EEE)\neq 0$ implies that $\widetilde{\NV}(\EEE) \neq 0$. The following statement is an easy consequence of \cite[Theorem A.3]{Xu21a}, which gives a sufficient condition on the converse.

\begin{thm}\label{thm Xu}
    Let $\EEE=(\ee_1,\ldots, \ee_n)$ be an admissible sequence of virtual extended $\Z$-segments. Suppose that 
    \[\Supp(\ee_1) \supseteq \Supp(\ee_2) \supseteq \cdots \supseteq  \Supp(\ee_n).\]
    Then $\NV(\EEE)\neq 0$ if and only if $\widetilde{\NV}(\EEE)\neq 0$.
\end{thm}

To verify that $\NV(\EEE)\neq 0$, it needs to check $\NV(\EEE')\neq 0$ for all $\EEE'\in[\EEE].$ Since $\EEE'$ is related to $\EEE$ by a composition of row exchanges, the following lemma is helpful for checking when $\NV(\EEE)\neq 0$. 

\begin{lemma}\label{lem exchange trans}
    Let $\ee_1, \ee_2, \ee_3$ be three extended $\Z$-segments such that
    $(\Delta_1, \Delta_2, \Delta_3)$, $(\Delta_2, \Delta_1,\Delta_3)$ and $(\Delta_1, \Delta_3,\Delta_2)$ are all admissible, where $\Delta_i= \Supp(\ee_i)$. Assume that $\ee_1 \preceq \ee_3$. Write
    \[  R_1(\ee_1,\ee_2,\ee_3)= (\ee_2', \ee_1',\ee_3),\ \ \ \  R_2(\ee_1,\ee_2,\ee_3)= (\ee_1, \ee_3',\ee_2'').\]
    Suppose that $\widetilde{\NV}(\ee_1,\ee_2,\ee_3) \neq 0$ so that $\ee_1',\ee_2',\ee_2'', \ee_3' \in \Eseg$ (Remark \ref{rmk row exchange}(2)). Then $\NV(\ee_1',\ee_3) \neq 0 $ if and only if $\NV(\ee_1,\ee_3')\neq 0$.
\end{lemma}
\begin{proof}
    For $i=1,2,3$, write $\ee_i=([A_i,B_i],l_i,\eta_i)$ and $\ee_i'=([A_i,B_i],l_i',\eta_i').$ 
   First, we consider the case that $[A_2,B_2]$ contains both $ [A_1,B_1] $ and $[A_3,B_3]$.      In this case, the formula for row exchanges implies that $(l_i',\eta_i')=(l_i,(-1)^{A_2-B_2}\eta_i)$ for $i=1,3$ and the conclusion holds.

   Next, we consider the case that $[A_2,B_2]$ contains one of $ [A_1,B_1] $ or $[A_3,B_3]$. By our assumption, this is exactly the case that
   $[A_1,B_1] \supseteq [A_2,B_2] \supseteq [A_3,B_3]$. Since $\widetilde{\NV}(\ee_1,\ee_2,\ee_3) \neq 0$, Theorem \ref{thm Xu} implies that $\NV(\ee_1,\ee_2,\ee_3) \neq 0$. Hence, $\NV(\ee_1',\ee_3) \neq 0 $ and $\NV(\ee_1,\ee_3')\neq 0$. 

   Finally, we consider the remaining case that $[A_2,B_2]$ is contained in both $ [A_1,B_1] $ and $[A_3,B_3]$. 
   If $A_2-B_2+1=2l_2$, then from Definition \ref{def row exchange}, we see that $\ee_1'=\ee_1$ and $\ee_3'=\ee_3$ and the conclusion trivially holds. Thus, we assume that $k:= A_2- B_2+1- 2l_2 >0$ from now on.  
    Since $\widetilde{\NV}(\ee_1,\ee_2,\ee_3) \neq 0$, by Definition \ref{def non-vanishing}, we must have $l_2 \leq l_1$ and $l_2 \leq l_3$ (note that $b_2-l_2\geq l_2$).
   Then it is a straightforward check from the definitions that the desired conclusion holds for the triple $(\ee_1,\ee_2,\ee_3)$ if and only if it holds for $(\widetilde{\ee_1},\widetilde{\ee_2},\widetilde{\ee_3})$, where
   \[ \widetilde{\ee_i}:= ([A_i- l_2, B_i+l_2], l_i-l_2, \eta_i ).\]
   Thus, we may further assume that $l_2=0$, which implies that $l_2'=0$ as well. It remains to prove the following claim: Let $\ee_2', \ee_1', \ee_3$ be three $\Z$-extended segments satisfying the following conditions.
   \begin{itemize}
       \item $\ee_2'=([A_2,B_2],0,\eta_2)$.
       \item $ \Supp(\ee_2') \subseteq \Supp(\ee_1')$, $\Supp(\ee_2') \subseteq \Supp(\ee_3)$.
       \item $\NV(\ee_2',\ee_1') \neq 0$.
       \item  Write $R_2 \circ R_1 (\ee_2', \ee_1',\ee_3)= (\ee_1,\ee_3',\ee_2'')$. Then $\ee_3'$ is a $\Z$-extended segment.
   \end{itemize}
   Then $\NV(\ee_1',\ee_3) \neq 0$ if and only if $\NV(\ee_1, \ee_3')\neq 0$.

   We apply induction on $k:= A_2-B_2+1$ to prove the claim. The case that $k=1$ follows from a straightforward case by case computation, which we omit. See Example \ref{exmp l=0 row exchange} for certain intuition for this case.  For $k>1$, we
   break $\ee_2'$ into two pieces: 
   \[ \ee_{2,1}:= ([B_2,B_2],0,\eta_2), \ \ \ee_{2,2}:= ([A_2,B_2+1],0,-\eta_2).\]
   One can check the following observation by a direct computation:   For any $\ee\in \Eseg$ such that $\Supp(\ee) \supseteq \Supp(\ee_2')$ and  $\NV(\ee_2',\ee)\neq 0$, write
   \[ R(\ee_2',\ee)=(\ee',\ee_2^{\ast}),\ \ R_{1}\circ R_2(\ee_{2,1},\ee_{2,2},\ee)= R_1(\ee_{2,1},\ee'', \ee_{2,2}^{\ast} )= (\ee''',\ee_{2,1}^{\ast}, \ee_{2,2}^{\ast} ).\]
  Then $\NV(\ee_{2,2},\ee) \neq 0$, $\NV(\ee_{2,1},\ee'') \neq 0$,  and most importantly, $\ee'''=\ee'$. Now write
   \[ R_3 \circ R_2( \ee_{2,1}, \ee_{2,2}, \ee_1', \ee_3)= (\ee_{2,1}, \ee_{1}'', \ee_3'', \ee_{2,2}^{\ast}), \ R_2 \circ R_1(\ee_{2,1}, \ee_{1}'', \ee_3'', \ee_{2,2}^{\ast})= (\ee_{1}''', \ee_{3}''', \ee_{2,1}^{\ast}, \ee_{2,2}^{\ast}). \]
   Then, by the above observation, we have $\ee_1'''=\ee$ and $\ee_3'''=\ee_3'$. By the induction hypothesis and the case that $k=1$, we have 
   \[ \NV(\ee_1',\ee_3 )\neq 0 \Longleftrightarrow \NV(\ee_1'',\ee_3' )\neq 0  \Longleftrightarrow \NV(\ee_1''',\ee_3''' )\neq 0.  \]
   This completes the proof of the claim and the lemma.
\end{proof}

Finally, we define a set connecting the non-vanishing of an admissible sequence of extended $\Z$-segments with an interval.

\begin{defn}\label{def EE_ee}
Let  $\Delta=[A,B]$ be a $\Z$-segment and $S \subseteq \Eseg_{\Delta}$ be an interval. Let $\EEE$ be an admissible sequence of extended $\Z$-segments such that $\NV(\EEE) \neq 0$. We define $\NV_{\EEE}(S)$ as follows.

    Write $\EEE^{(P')}=(\ee_1 ,\ldots, \ee_n)$ and $\Supp(\ee_i)=[A_i,B_i]$. Let 
\[j:= \min(\{ 1 \leq i \leq n \ | \ B_j>B \} \cup \{n+1\}).\]
For any $\ee \in \Eseg_{\Delta}$, we let $\ee^\dagger$ be  the unique member in $\Eseg_{\Delta}$ such that $\NV(\ee,\ee^{\dagger})\neq 0$ (see Remark \ref{rmk non-vanishing}(2)). Then we define
\[ \EEE_{\ee}:= (\ee_1 ,\ldots, \ee_{j-1}, \ee, \ee^{\dagger} , \ee_{j},\ldots, \ee_{n} ). \]
Finally, we define
\begin{align*}
   \NV_{\EEE}(S)= \{ \ee \in S \ | \ \NV( \EEE_{\ee} )\neq 0  \}.
\end{align*}
\end{defn}

We list some properties  for $\ee^{\dagger}$ as follows.

\begin{lemma}\label{lem e dagger}Let $\ee, \ee_1, \ee_2$ be extended $\Z$-segments. 
    \begin{enumerate}
            \item We have that $R(\ee, \ee^{\dagger})= (\ee,\ee^{\dagger})$.
            \item We have that $\NV(\ee_1, \ee_2) \neq 0$ if and only if $ \NV( \ee_1, \ee_2 , \ee_2^{\dagger}) \neq 0$.  Similarly, $\NV(\ee_1^{\dagger}, \ee_2) \neq 0$ if and only if $ \NV( \ee_1, \ee_1^{\dagger} , \ee_2) \neq 0.$
            \item If $R(\ee_1,\ee_2)= (\ee_2', \ee_1')$, then 
            \[ R_2 \circ R_{1} (\ee_1,\ee_2, \ee_2^{\dagger})= R_2(\ee_2', \ee_1' , \ee_2^{\dagger})= (\ee_2', (\ee_2')^{\dagger}, \ee_1). \]
        \end{enumerate}
        \end{lemma}
        \begin{proof}
            These statements follow from  direct computations, which we omit.
        \end{proof}
    
The following proposition will be used to decompose a certain parabolic induction in later sections (see Theorems \ref{thm reducibility unitary induction} and \ref{thm reduction}).

    \begin{prop}\label{prop NV segment}
   Suppose $\NV(\EEE) \neq 0$ and let $S \subseteq \Eseg_{\Delta}$, $S' \subseteq \Eseg_{\Delta'}$ be two intervals.
    \begin{enumerate}
        \item Then, the set $\NV_{\EEE}(S)$ is an interval.
        \item Suppose that the following conditions hold:
        \begin{enumerate}
            \item [(a)]$\Delta \preceq \Delta'$ and $\Delta \neq \Delta'$.
            \item [(b)] $|\NV_{\EEE}(S)|>1$ and $|\NV_{\EEE}(S')|>1$.
            \item [(c)] $\ee_1^*, \ee_2^* \in \NV_{\EEE}(S)$ are adjacent, $|\NV_{\EEE_{\ee_1^*}}(S')|=1$ and $|\NV_{\EEE_{\ee_2^*}}(S')|>0$.
        \end{enumerate}
        Then $|\NV_{\EEE_{\ee_2^*}}(S')|=2$.
    \end{enumerate}
    \end{prop}
\begin{proof} 
Let us set up some notation for this proof. For any $\ee \in S$, write 
\[\EEE_{\ee}= (\ee_1,\ldots, \ee_{j-1}, \ee_j=\ee, \ee_{j+1}=\ee^{\dagger}, \ee_{j+2}, \ldots, \ee_{n}).\]
Any member of $ [\EEE_{\ee}]$ is of the form
\begin{align}\label{eq EE_ee,sigma}
    \EEE_{\ee,\sigma}= (\widetilde{\ee}_{\sigma^{-1}(1)},\ldots, \widetilde{\ee}_{\sigma^{-1}(n)}  )= R_{k_1}\circ \cdots \circ R_{k_{s}} (\EEE).
\end{align}
Here $\sigma$ is a permutation of $\{1,\dots, n\}$, which is the product of transpositions given by each $R_{k_i}$. In particular, the extended $\Z$-segment $\ee_i$ is sent to $\widetilde{\ee}_{\sigma(i)}$ under this row exchanges process. We let $[\EEE_{\ee}]^{\dagger}$ denote the subset of $ \EEE_{\ee, \sigma} \in [\EEE_{\ee}]$ such that $ \sigma(j+1)- \sigma(j)=1$. Namely, $\ee$ and $\ee^{\dagger}$ are still neighbors in $\EEE_{\ee, \sigma}$ after the permutation.

We first give two observations of Lemma \ref{lem e dagger}.
\begin{enumerate}
    \item [(i)] $\NV(\EEE_{\ee}) \neq 0$ if and only if $\widetilde{\NV}( \EEE_{\ee,\sigma}) \neq 0$ for all $ \EEE_{\ee,\sigma} \in [\EEE_{\ee}]^{\dagger}$.
    \item [(ii)] Suppose $\EEE_{\ee,\sigma} \in [\EEE_{\ee}]^{\dagger}$. We let $\EEE_{\sigma}$ denote the sequence obtained from $\EEE_{\ee,\sigma} $ by deleting $\ee_{\sigma(j)}'$ and $\ee_{\sigma(j+1)}'$. Then $\EEE_{\sigma} \in [\EEE]$ and hence $\widetilde{\NV}(\EEE_{\sigma}) \neq 0$. 
\end{enumerate}
Indeed, for Observation (i), Parts (2) and (3) of Lemma \ref{lem e dagger} guarantee that if $\widetilde{\NV}(\EEE_{\ee, \sigma} ) =0$, then we may obtain another $\EEE_{\ee, \sigma'} \in [\EEE_{\ee}]^{\dagger}$ such that $\widetilde{\NV}(\EEE_{\ee, \sigma'} ) =0$. Observation (ii) is a direct consequence of Lemma \ref{lem e dagger}(3).

Now we prove Part (1), which is equivalent to showing the following claim: Suppose $\ee_\alpha \neq \ee_\beta \in \NV_{\EEE}(S)$ and let $S_{\alpha,\beta}$ be the smallest interval containing $\ee_\alpha$ and $\ee_\beta$. Then for any $\ee_\gamma \in S_{\alpha,\beta}$, we must have $\NV(\EEE_{\ee_\gamma})\neq 0$ as well. 

Take any $\ee_\gamma \in S_{\alpha,\beta}$. By Observation (i), it suffices to show that  $\widetilde{\NV}(\EEE_{\ee_\gamma,\sigma})\neq 0$ for any $\EEE_{\ee_\gamma, \sigma} \in [\EEE_{\ee_\gamma}]^{\dagger}$.  Note that $\ee_\gamma^{\dagger}$ also lies $S_{\alpha^{\dagger}, \beta^{\dagger}}$, the smallest interval containing $\ee_{\alpha}^{\dagger}$ and $\ee_{\beta}^{\dagger}$. For $\delta\in \{\alpha, \beta\}$, since $\NV(\EEE_{\ee_\delta})\neq 0$, we have $\widetilde{\NV}(\EEE_{\ee_\delta,\sigma})\neq 0$. Say $\sigma(j)=k$, $\sigma(j+1)=k+1$ and write 
\begin{align*}
    \EEE_{\ee_{\delta},\sigma}= (\widetilde{\ee}_{\sigma^{-1}(1)},\ldots,\widetilde{\ee}_{\sigma^{-1}(k-1)}, \widetilde{\ee}_{\delta}, \widetilde{\ee}_{\delta}^{\dagger}, \widetilde{\ee}_{\sigma^{-1}(k+2)}, \ldots,\widetilde{\ee}_{\sigma^{-1}(n)}),
\end{align*}
 where $\delta\in \{\alpha, \beta, \gamma\}$. Note that here $\widetilde{\ee}_{\sigma^{-1}(i)}$ for $i \not\in \{k,k+1\}$ is independent of $\delta$ by Lemma \ref{lem e dagger}(3).
 By Observation (ii), we have $\widetilde{\NV}(\EEE_{\sigma})\neq 0$, where
 \[ \EEE_{\sigma}= (\widetilde{\ee}_{\sigma^{-1}(1)},\ldots,\widetilde{\ee}_{\sigma^{-1}(k-1)}, \widetilde{\ee}_{\sigma^{-1}(k+2)}, \ldots,\widetilde{\ee}_{\sigma^{-1}(n)}) \in [\EEE]. \]
 Thus, $\widetilde{\NV}(\EEE_{\gamma,\sigma})\neq 0$ if and only if $\NV( \widetilde{\ee}_{\sigma^{-1}(k-1)}, \widetilde{\ee}_{\gamma})\neq  0$ and $\NV(\widetilde{\ee}_{\gamma}^{\dagger}, \widetilde{\ee}_{\sigma^{-1}(k+2)})\neq 0$. By Lemma \ref{lem observation row exchange}, $\widetilde{\ee}_{\gamma}$ must lie in the smallest interval containing $ \widetilde{\ee}_{\alpha}$ and $\widetilde{\ee}_{\beta}$. On the other hand, Lemma \ref{lem adj} implies that $\NV_{(\widetilde{\ee}_{\sigma^{-1}(k-1)},-)}(\Delta)$ is an interval containing $ \widetilde{\ee}_{\alpha}$ and $\widetilde{\ee}_{\beta}$. Therefore, we conclude that $\widetilde{\ee}_{\gamma} \in \NV_{(\widetilde{\ee}_{\sigma^{-1}(k-1)},-)}(\Delta)$ and $\NV( \widetilde{\ee}_{\sigma^{-1}(k-1)}, \widetilde{\ee}_{\gamma})\neq  0$. A similar argument shows that $\NV(\widetilde{\ee}_{\gamma}^{\dagger}, \widetilde{\ee}_{\sigma^{-1}(k+2)})\neq 0$.  This completes the proof of the claim and Part (1).

Now we prove Part (2). For $\ee \in \NV_{\EEE}(S)$ and $\ee' \in \NV_{\EEE}(S')$, write 
\[\EEE_{\ee,\ee'}:= (\EEE_{\ee})_{\ee'} = (\ee_1,\ldots, \ee_{j}=\ee, \ee_{j+1}=\ee^{\dagger}, \ldots,\ee_{j'}=\ee', \ee_{j'+1}=(\ee')^{\dagger}, \ldots, \ee_{n}),\]
and define $\EEE_{\ee,\ee', \sigma}:= (\EEE_{\ee})_{\ee',\sigma}$ as in \eqref{eq EE_ee,sigma}. Let 
\begin{align*}
    [\EEE_{\ee,\ee'}]^{\dagger}&:= \{ \EEE_{\ee,\ee', \sigma} \in [ \EEE_{\ee,\ee'}]\ | \ \sigma(j+1)-\sigma(j)=1 \text{ and }\sigma(j'+1)-\sigma(j')=1  \},\\
    [\EEE_{\ee,\ee'}]^{\dagger\dagger}&:= \{ \EEE_{\ee,\ee', \sigma} \in [ \EEE_{\ee,\ee'}]^{\dagger}\ | \ \sigma(j')=\sigma(j+1)+1  \}.
\end{align*}
It is possible that $[\EEE_{\ee,\ee'}]^{\dagger\dagger}$ is empty.

Observations (i) and (ii) imply that $ \NV(\EEE_{\ee,\ee'})\neq 0$ if and only if $(\ee,\ee') \in \NV_{\EEE}(S)\times \NV_{\EEE}(S')$ and $\widetilde{\NV}(\EEE_{\ee,\ee',\sigma})\neq 0$ for any $\EEE_{\ee,\ee',\sigma} \in [\EEE_{\ee,\ee'}]^{\dagger\dagger}$. Suppose that $[\EEE_{\ee,\ee'}]^{\dagger\dagger}$ is empty for some (equivalently, for all) $(\ee,\ee') \in \NV_{\EEE}(S)\times \NV_{\EEE}(S')$. Then $\NV_{\EEE_{\ee}}(S')= \NV_{\EEE}(S')$ for any $\ee \in \NV_{\EEE}(S)$. This contradicts Conditions (b) and (c). Thus, $[\EEE_{\ee,\ee'}]^{\dagger\dagger}$ must be non-empty for some $(\ee,\ee') \in \NV_{\EEE}(S)\times \NV_{\EEE}(S')$. By Lemmas \ref{lem exchange trans}, \ref{lem e dagger}(3), and the above discussion, the following three statements are equivalent.
\begin{itemize}
    \item $\widetilde{\NV}( \EEE_{\ee,\ee',\sigma}) \neq 0$ for a fixed $\EEE_{\ee,\ee',\sigma} \in [\EEE_{\ee,\ee'}]^{\dagger\dagger}$. 
    \item $\widetilde{\NV}( \EEE_{\ee,\ee',\sigma}) \neq 0$ for any $\EEE_{\ee,\ee',\sigma} \in [\EEE_{\ee,\ee'}]^{\dagger\dagger}$.
    \item $\NV(\EEE_{\ee, \ee'})\neq0$.
\end{itemize}

Now we fix an $\EEE_{\ee,\ee',\sigma} \in [\EEE_{\ee,\ee'}]^{\dagger\dagger}$ and write $(\sigma(j), \sigma(j+1), \sigma(j'), \sigma(j'+1))=(k,k+1,k+2,k+3)$, 
\[ \EEE_{\ee,\ee',\sigma}= (\widetilde{\ee}_{\sigma^{-1}(1)},\ldots, \widetilde{\ee}_{\sigma^{-1}(k-1)}, \widetilde{\ee},\widetilde{\ee}^{\dagger}, \widetilde{\ee'}, \widetilde{\ee'}^{\dagger}, \widetilde{\ee}_{\sigma^{-1}(k+4)}, \ldots,\widetilde{\ee}_{\sigma^{-1}(n)}  )=: \EEE_{\widetilde{\ee},\widetilde{\ee'}}. \]
Note that $\widetilde{\ee}$ only depends on $\ee$ and $\sigma$ but not on $\ee'$ and similarly for $\widetilde{\ee'}$. By Lemma \ref{lem observation row exchange}, the intervals $S, S'$ are sent to the virtual intervals $S_{\sigma}, S_{\sigma'}$ under the composition of row exchanges from $\EEE_{\ee,\ee'}$ to $\EEE_{\ee,\ee',\sigma}$. Then the intervals $\NV_{\EEE}(S)$ and $ \NV_{\EEE}(S')$ are sent to intervals $ T_{\sigma}:= S_{\sigma} \cap \NV_{\sigma}$ and $T_{\sigma}':=S_{\sigma}' \cap \NV_{\sigma}'$, where $\NV_{\sigma}$ and $\NV_{\sigma}'$ are the following intervals
\begin{align*}
  \NV_{\sigma}&:= \{\widetilde{\ee} \in \Eseg_{\Delta}\ | \ \NV(\widetilde{\ee}_{\sigma^{-1}(1)},\ldots, \widetilde{\ee}_{\sigma^{-1}(k-1)}, \widetilde{\ee},\widetilde{\ee}^{\dagger},\widetilde{\ee}_{\sigma^{-1}(k+4)}, \ldots,\widetilde{\ee}_{\sigma^{-1}(n)}  )\neq 0 \},\\
  \NV_{\sigma}'&:=\{\widetilde{\ee'} \in \Eseg_{\Delta'}\ | \ \NV(\widetilde{\ee}_{\sigma^{-1}(1)},\ldots, \widetilde{\ee}_{\sigma^{-1}(k-1)}, \widetilde{\ee'}, \widetilde{\ee'}^{\dagger},\widetilde{\ee}_{\sigma^{-1}(k+4)}, \ldots,\widetilde{\ee}_{\sigma^{-1}(n)}  )\neq 0 \}.
\end{align*}
Again, Observation (ii) implies that $\widetilde{\NV}( \EEE_{\ee,\ee',\sigma}) \neq 0$ if and only if $\NV( \widetilde{\ee}^{\dagger}, \widetilde{\ee'}) \neq 0$. We conclude that $\NV(\EEE_{\ee,\ee'})\neq 0$ (and $\ee\in S, \ee' \in S'$) if and only if $\widetilde{\ee} \in T_{\sigma} $ and $\widetilde{\ee'} \in \NV_{(\widetilde{\ee}^{\dagger},-)}(\Delta'
) \cap T_{\sigma}'$. In other words, if we write $\widetilde{T_{\widetilde{\ee}}}:= \NV_{(\widetilde{\ee}^{\dagger},-)}(\Delta'
)$, then
\[ |\NV_{\EEE_{\ee}}(S')|= | \widetilde{T_{\widetilde{\ee}}} \cap T_{\sigma}'|.\]

Finally, we take the Conditions (a), (b) and (c) into consideration. These conditions and Lemma \ref{lem adj}(iii) imply the following: 
\begin{itemize}
    \item $|\widetilde{T}_{\widetilde{\ee}_1^*}| >1$ and $|T_{\sigma}'|=| \NV_{\EEE}(S')| >1$;
    \item $\widetilde{T}_{\widetilde{\ee}_1^*}$ and $\widetilde{T}_{\widetilde{\ee}_2^*}$ are adjacent, $|\widetilde{T}_{\widetilde{\ee}_1^*} \cap T_{\sigma}'|=1 $ and $|\widetilde{T}_{\widetilde{\ee}_2^*} \cap T_{\sigma}'|>0 $.
\end{itemize}
Thus, Lemma \ref{lem key observation}(c) implies that $|\NV_{\EEE_{\ee_2^*}}(\Delta')|=|\widetilde{T}_{\widetilde{\ee}_2^*} \cap T_{\sigma}'|=2.$ This completes the proof of the proposition.
\end{proof}

\section{Extended multi-segments}\label{sec extended multi-segments}

In a series of papers (\cite{Moe06a, Moe06b, Moe09a, Moe10, Moe11a}), M{\oe}glin explicitly constructed each local Arthur packet $\Pi_{\psi}$ and showed that it is multiplicity free. In recent work (\cite{Ato20b}), 
Atobe gave a reformulation on  M{\oe}glin's construction. 
In this section, we recall certain related results, mainly from \cite{Ato20b, Ato22d}, using the notation developed in previous sections.

\subsection{Extended multi-segments and associated representations}\label{sec extended multi-segments and rep}
We start with the definition of extended multi-segments.

\begin{defn}[Extended multi-segments]\label{def multi-segment}\ 

\begin{enumerate}
\item
An \emph{extended segment} is a triple $([A,B]_\rho, l, \eta)$,
where
\begin{itemize}
\item
$[A,B]_\rho = \{\rho\lvert \cdot \rvert^A, \rho\lvert \cdot \rvert^{A-1}, \dots, \rho\lvert \cdot \rvert^B \}$ is a segment 
for an irreducible unitary supercuspidal representation $\rho$ of some $\GL_d(F)$ with $A \in \half{\Z}$; 
\item
$l \in \Z$ with $0 \leq l \leq \frac{b}{2}$, where $b = \#[A,B]_\rho = A-B+1$; 
\item
$\eta \in \{\pm1\}/E$, where $E=\{\pm 1\}$ if $b=2l$ and $E=\{+1\}$ if $b>2l$. 
\end{itemize}
\item Consider a multi-set of extended segments of the form $$\{([A_i,B_i]_{\rho},l_i,\eta_i)\}_{i \in I_{\rho}}.$$
We say that a total order $>$ on $I_{\rho}$ is admissible (or satisfies (P)) if
\[ A_i< A_j, B_i< B_j\Longrightarrow i<j. \]
We say that an admissible order $>$ satisfies ($P'$) if
\[  B_i< B_j\Longrightarrow i<j. \]

\item
An \emph{extended multi-segment} for $G_n$ is a union of multi-sets of extended segments indexed by collection of total ordered sets $(I_{\rho},>)$
\[
\EE = \cup_{\rho}\{ ([A_i,B_i]_{\rho}, l_i, \eta_i) \}_{i \in (I_\rho,>)}
\]
such that 
\begin{enumerate}
\item
$I_\rho$ is a totally ordered finite set with a fixed total order $>$ satisfies (P);

\item
$A_i + B_i \geq 0$ for all $\rho$ and $i \in I_\rho$; 

\item
as a representation of $W_F \times \SL_2(\BC) \times \SL_2(\BC)$, 
\[
\psi_{\EE} := \bigoplus_\rho \bigoplus_{i \in I_\rho} \rho \otimes S_{a_i} \otimes S_{b_i} 
\]
where $(a_i, b_i) = (A_i+B_i+1, A_i-B_i+1)$,
is a local Arthur parameter for $G_n$ of good parity.
\item The following sign condition holds
\begin{align*}
\prod_{\rho} \prod_{i \in I_\rho} (-1)^{[\frac{b_i}{2}]+l_i} \eta_i^{b_i} = 1.
\end{align*}
\end{enumerate}
\item Suppose $\psi_{\EE}=\psi_{\EE'}$. We say $\EE$ and $\EE'$ have the same admissible order if they are of the form 
\[\EE= \cup_{\rho} \{([A_i, B_i]_{\rho},l_i,\eta_i)\}_{i \in (I_{\rho},>)}, \ \EE'= \cup_{\rho} \{([A_i, B_i]_{\rho},l_i',\eta_i')\}_{i \in (I_{\rho},>)}. \]
    Namely, after identifying the total ordered index sets $(I_{\rho},>)$, the underlying segments $[A_i,B_i]_{\rho}$ are identical for all $i \in I_{\rho}$.
\end{enumerate}
\end{defn}

The extended multi-segments and local Arthur packets are related by the following theorem (\cite[Section 3]{Ato20b}).

\begin{thm}\label{thm Atobe}
    Let $\EE$ be an extended multi-segment and $\psi$ a local Arthur parameter of $G_n$.
    \begin{enumerate}
        \item  There is a representation $\pi(\EE)$ of $G_n$ associated to $\EE$, which is either irreducible or zero.
        \item Suppose $\psi= \bigoplus_{\rho} \bigoplus_{i \in I_{\rho}} \rho \otimes S_{a_i} \otimes S_{b_i}$ is a local Arthur parameter of $G_n$ of good parity. Fix an admissible order $>$ on $I_{\rho}$ for each $\rho$ that satisfies ($P'$) if $\half{a_i-b_i}<0$ for some $i \in I_{\rho}$. Then
\[ \bigoplus_{\pi \in \Pi_{\psi}} \pi= \bigoplus_{\EE} \pi(\EE),\]
where $\EE$ runs over all extended multi-segments with $\psi_{\EE}=\psi$ and $\pi(\EE) \neq 0$, and the total orders are the ones fixed above.
    \end{enumerate}
\end{thm}

\begin{remark}\label{rmk multiplicity one}
   M{\oe}glin's result that local Arthur packets are multiplicity-free \textup{(\cite{Moe11a})} implies the following statement: Suppose $\psi_{\EE}=\psi_{\EE'}$ and $\EE, \EE'$ have the same admissible order (Definition \ref{def multi-segment}(4)). Then $\pi(\EE)=\pi(\EE')$ if and only if $\EE=\EE'$.
\end{remark}

We remark that, in this paper, we do not need the precise computation of $\pi(\EE)$. We only need the followings, which are recalled in the following subsections. 

    \begin{itemize}
        \item A criterion on $\EE$ that $\pi(\EE)$ give nonzero representations. (See \S \ref{sec non-vanishing extended multi-segments}.)
        \item A recipe to compute the character associated to $\pi(\EE)$ in $\widehat{\mathcal{S}}_{\psi_{\EE}}$. (See \S \ref{sec character extended multi-segment}.)
        \item A formula for the decomposition of unitary induction of representations of Arthur type and good parity in terms of extended multi-segments (see \S \ref{sec unitary induction of Arthur type} for more explanation). 
    \end{itemize}

Before presenting these results, to each extended multi-segment $\EE$, we associate several admissible sequences of extended $\Z$-segments. We then extend the definitions in \S \ref{sec Z-extended segments} for extended $\Z$-segments to the setting of extended multi-segments. 

First we consider extended segments and extended $\Z$-segments.
\begin{defn}\label{def transport extended segments}\ 
    
    \begin{enumerate}
    \item Suppose $\ee= ([A,B]_{\rho}, l, \eta)$ is an extended segment. We associate an extended $\Z$-segment $\ee^{\Z}$ by
    \[ \ee^{\Z}:= ([\lfloor A\rfloor,\lfloor B\rfloor], l, \eta). \]
        \item We say two extended segments $\ee_1, \ee_2$ are adjacent if $ \ee_1^{\Z}, \ee_2^{\Z}$ are adjacent.
        \item Fixing a segment $\Delta=[A,B]_{\rho}$, we let $\Eseg_{\Delta}$ denote the set of extended segments with underlying segments $[A, B]_{\rho}$.
        
        \item Let $S=\{\ee_1,\ldots, \ee_r\} \subseteq \Eseg_{[A,B]_{\rho}}$ be a finite set of extended segments. We define $S^{\Z}:= \{\ee_1^{\Z},\ldots, \ee_{r}^{\Z}\}$, and we say $S$       
        is an interval if $ S^{\Z}$ is an interval.
    \end{enumerate}
\end{defn}

Next, we consider extended multi-segments.

\begin{defn}\label{def transport extended multi-segments}
Let $\EE=\cup_{\rho}\{ \ee_i \}_{i \in (I_{\rho},>)}$ be an extended multi-segment. We identify the total ordered set $(I_\rho,>)$ as $\{1 < \cdots < |I_{\rho}|\}.$
\begin{enumerate}
    \item For each $\rho$, we associate an admissible sequence of extended $\Z$-segments $\EE_{\rho}^{\Z}$ by
    \[ \EE_{\rho}^{\Z}:= (\ee_1^{\Z},\ldots, \ee_{|I_{\rho}|}^{\Z}).\]
    \item  Write $\ee_i= ([A_i,B_i]_{\rho},l_i,\eta_i)$.    For $1 \leq k \leq n-1$ such that $[A_k,B_k]_{\rho} \supseteq [A_{k+1},B_{k+1}]_{\rho}$ or $[A_k,B_k]_{\rho} \subseteq [A_{k+1},B_{k+1}]_{\rho}$, we define an extended multi-segment $R_{k}(\EE)$ by requiring that
    \[ (R_k(\EE))_{\rho'}^{\Z}= \begin{cases}
        \EE_{\rho'}^{\Z} & \text{ if }\rho' \not\cong \rho,\\
        R_{k}( \EE_{\rho}^{\Z}) & \text{ if }\rho'\cong \rho.
    \end{cases}\]
    \item We write $\EE\stackrel{R}{=} \EE'$ if $\EE_{\rho}^{\Z}\stackrel{R}{=} (\EE')_{\rho}^{\Z}$ for every $\rho$. We let $[\EE]:= \{\EE' \ | \ \EE' \stackrel{R}{=} \EE \}$.
    \item We define $\NV(\EE) \neq 0$ if and only if $\NV(\EE_{\rho}^{\Z}) \neq 0$ for every $\rho $.
    \item Suppose that $\psi_{\EE}+(\rho \otimes S_{a} \otimes S_b)^{\oplus 2}$ is a local Arthur parameter of good parity of the same type of group. Let $A:= \half{a+b}-1, B:= \half{a-b}$. For any $\ee \in \Eseg_{[A,B]_{\rho}}$, we define an extended multi-segment $\EE_{\ee}$ by requiring that
    \[ (\EE_{\ee})_{\rho'}^{\Z}= \begin{cases}
        \EE_{\rho'}^{\Z} & \text{ if }\rho' \not\cong \rho,\\
        (\EE_{\rho}^{\Z})_{\ee^{\Z}} & \text{ if }\rho'\cong \rho.
    \end{cases}\]
    Here $(\EE_{\rho}^{\Z})_{\ee^{\Z}}$ is the admissible sequence of extended $\Z$-segment defined in Definition \ref{def EE_ee}.
    \item To simplify notation, for an extended multi-segment $\EE$ and an interval $S \subseteq \Eseg_{[A,B]_{\rho}}$ of extended segments, we let
    \[ \NV_{\EE}(S):= \{\ee \in \Eseg_{[A,B]_{\rho}}\ | \ \ee^{\Z} \in \NV_{\EE_{\rho}^{\Z}}(S^{\Z}) \}. \]
\end{enumerate}
\end{defn}

We conclude this subsection by recalling a fact on row exchanges.

\begin{prop}[{\cite[Theorem 4.3]{Ato20b}}]\label{prop row exchanges}
    Suppose $\pi(\EE)\neq 0$ and $\psi_{\EE}=\psi_{\EE'}$. Then $\pi(\EE)=\pi(\EE')$ if and only if $\EE' \in [\EE]$.
\end{prop}

\subsection{Non-vanishing criterion}\label{sec non-vanishing extended multi-segments}
Xu (\cite{Xu21b}) gave an algorithm to determine whether the representations in M{\oe}glin's construction are nonzero.
Atobe reformulated this algorithm on non-negative extended multi-segments (\cite[\S 4]{Ato20b}).
We recall this below.

\begin{thm}[{\cite[Theorems 3.7, 4.4]{Ato20b}}]\label{thm non-vanishing}
Suppose $\EE=\cup_{\rho}\{ ([A_i,B_i]_{\rho}, l_i, \eta_i)\}_{i \in (I_{\rho},>)}$ is an extended multi-segment. Then $\pi(\EE)\neq 0$ if and only if $\NV(\EE) \neq 0$ and the following condition holds for any $\rho$ and $i \in I_{\rho}$:
\begin{align}\label{eq ast}
           B_i+l_i \geq \begin{cases}  0 & \text{ if }B_i \in \Z, \\
(-1)^{\alpha_i+1}\eta_i \cdot  \half{ 1} & \text{ if } B_i \not\in \Z ,
        \end{cases}
\end{align}
        where 
        \[ \alpha_i:= \sum_{j < i }A_j+B_j+1. \]
\end{thm}

\begin{remark}\label{rmk interval non-vanishing ineq}
    Fix an admissible order $>$ on $I_{\rho}$. For each $i\in I_{\rho}$, the set of extended segments $([A_i,B_i]_{\rho}, l_i ,\eta_i )$ satisfying the inequality \eqref{eq ast} is an interval.
\end{remark}

Now we prove Lemma \ref{lem P''} stated in the previous section.
\begin{proof}[Proof of Lemma \ref{lem P''}]
First, we set up several notations. Let $\EEE= (\ee_1,\ldots, \ee_n)$ be an admissible sequence of $\Z$-segments with $\ee_i=([A_i,B_i],l_i,\eta_i)$. For $t \in \Z$, we define $sh^t(\EEE):= (\ee_1^t, \ldots, \ee_n^t)$, where
\[ \ee_i^t:= ([A_i+t,B_i+t],l_i,\eta_i).\]
We say $\EEE$ is positive if $B_i >0$ for every $1 \leq i \leq n.$ The following observations are clear from the definitions.
\begin{itemize}
    \item Fixing $\EEE$, $sh^t(\EEE)$ is positive for some large enough $t$.
    \item If $R_{k}(\EEE)$ is well-defined, then so is $R_{k}(sh^t(\EEE))$ and $R_{k}(sh^t(\EEE))=sh^t(R_k(\EEE))$.
    \item We have $\widetilde{\NV}(\EEE)\neq 0$ if and only if $\widetilde{\NV}(sh^t(\EEE))\neq 0$. Moreover, ${\NV}(\EEE)\neq 0$ if and only if ${\NV}(sh^t(\EEE))\neq 0$.
    \item The map $\EEE' \mapsto sh^t(\EEE')$ is a bijection between $[\EEE]$ and $[sh^t(\EEE)]$ that sends $[\EEE]^{(P'')}$ to $[sh^t(\EEE)]^{(P'')}$.
\end{itemize}

Let $\EEE$ be an admissible sequence of $\Z$-segments such that $\NV(\EEE)\neq 0$. By the above observations, we may assume $\EEE$ is positive by replacing $\EEE$ with $sh^t(\EEE)$ for some large enough $t$ if necessary. There exists an extended multi-segment $\EE= \EE_{\rho_1} \cup \EE_{\rho_2}$ of $\Sp_{2n}(F)$ or split $\SO_{2n+1}(F)$ such that $\EE_{\rho_1}^{\Z}= \EEE$, $\NV(\EE_{\rho_2}^{\Z}) \neq 0$, and $\EE_{\rho_2}^{\Z}$ is positive. Then Theorem \ref{thm non-vanishing} implies that $\pi(\EE) \neq 0$.

Suppose $\EEE'=(\ee_1',\ldots, \ee_n')$ and $ \EEE'' =(\ee_1'',\ldots, \ee_n'')$ both lie in $ [\EEE]^{(P'')}$. By the definition of $(P'')$, we see that for every $1 \leq i \leq n$,
\begin{align}\label{eq pf of P''}
    \Supp(\ee_i')=\Supp(\ee_i'').
\end{align} 
Define extended multi-segments $\EE':=\EE_{\rho_1}' \cup \EE_{\rho_2}$ and $\EE':=\EE_{\rho_1}'' \cup \EE_{\rho_2}$ such that
\[ (\EE_{\rho}')^{\Z}=\EEE',\ (\EE_{\rho}'')^{\Z}=\EEE''. \]
Then it follows from the definition that $\EE', \EE'' \in [\EE]$ and hence $\pi(\EE')=\pi(\EE'')$ by Proposition \ref{prop row exchanges}. Therefore, by \eqref{eq pf of P''} and the fact that local Arthur packets are multiplicity-free (see Remark \ref{rmk multiplicity one}), we must have $\EE'=\EE''$. We conclude that $\EEE'=\EEE''$, which completes the proof of the lemma.
\end{proof}

Finally, we prove Theorem \ref{thm Xu}.
\begin{proof}[Proof of Theorem \ref{thm Xu}]
Write $\Supp(\ee_1)= [A, B]$ and take $t\in \Z_{>0} $ such that $B+t >0$. For any $\EEE' = (\ee_1',\ldots, \ee_n')\in [\EEE]$, where $\ee_i'= ([A_i',B_i'], l_i', \eta_i')$, we define $\overline{\EEE'}:= (\overline{\ee_1'},\ldots, \overline{\ee_n'})$ by
\[ \overline{\ee_i'}:= ([A +t, B_i' - A_i'+A+t ], l_i', \eta_i').\]
It is clear that $\overline{\EEE'}$ is also an admissible sequence of virtual extended $\Z$-segments. Moreover, this gives us a map
\begin{align*}
    [\EEE] &\to [\overline{\EEE}],\\
    \EEE' & \mapsto \overline{\EEE'},
\end{align*}
which is indeed a bijection. It is clear from Definitions \ref{def non-vanishing}, \ref{def non-vanishing EEE} that for any $\EEE' \in [\EEE]$, $\widetilde{\NV}(\EEE')\neq 0$ if and only $\widetilde{\NV}(\overline{\EEE'})\neq 0$. Hence, $\NV(\EEE)\neq 0$ if and only $\NV(\overline{\EEE'})\neq 0$.

Now we take any extended multi-segment $\EE$ such that 
\begin{enumerate}
    \item for a fixed $\rho$, we have $\EE_{\rho}^{\Z}= \overline{\EEE}$, and
    \item for $\rho' \not \cong \rho$,  either $\EE_{\rho'}$ is empty or     
    $\EE_{\rho'}= \{([A_{\rho'},B_{\rho'}]_{\rho'}, l_{\rho'}, \eta_{\rho'})\}$ is a singleton with $B_{\rho'}>0$. 
\end{enumerate}
Then Theorem \ref{thm non-vanishing} implies that $\pi(\EE) \neq 0$ if and only if $\NV(\overline{\EEE})\neq 0$. On the other hand, \cite[Theorem A.3]{Xu21a} implies that $\pi(\EE)\neq 0$ if and only if $\widetilde{NV}(\overline{\EEE}) \neq 0$. In conclusion, we have
\[ \widetilde{\NV}(\EEE)\neq 0 \Leftrightarrow \widetilde{\NV}(\overline{\EEE})\neq 0 \Leftrightarrow \pi(\EE)\neq 0 \Leftrightarrow \NV(\overline{\EEE})\neq 0  \Leftrightarrow \NV(\EEE)\neq 0 .\]
This completes the proof of the theorem.
\end{proof}

\subsection{Characters of component group}\label{sec character extended multi-segment}
Recall that there is a map
\begin{align*}
    \Pi_{\psi} &\to \widehat{\mathcal{S}}_{\psi},\\
    \pi& \mapsto \langle \cdot, \pi\rangle_{\psi}.
\end{align*}
A combinatorial description of $\widehat{\mathcal{S}}_\psi$ is recalled in \S \ref{sec character}.
Suppose $\psi$ is of good parity. Then every member of $\Pi_{\psi}$ is of the form $\pi(\EE)$ with $\psi_{\EE}=\psi$. We recall the recipe to compute $\langle \cdot, \pi(\EE)\rangle_{\psi_{\EE}}$.

\begin{defn}
Let $\EE= \cup_{\rho}\{([A_i,B_i]_{\rho}, l_i, \eta_i)\}_{i \in (I_{\rho,>})}$. We assume that the admissible orders of $\EE$ satisfy ($P'$).  Let $a_i:= A_i+B_i+1$ and $b_i:=A_i-B_i+1$. 
\begin{enumerate}
    \item For $i \in I_{\rho}$, define $Z(\EE)_{\rho,i}$ by the set of $j \in I_{\rho}$ such that
    \begin{itemize}
        \item $b_i \not\equiv b_j \mod 2$ (so that $a_i \not\equiv a_j \mod 2$);
        \item $j>i$ and $a_j <a_i$ or $j<i$ and $a_j>a_i$;
        \item $b_j \in 2 \Z $ and $b_j >b_i$, or  $b_j \not\in 2 \Z $ and $b_i >b_j$.
    \end{itemize}
    \item We define the character $\eta_{\EE} \in \widehat{\mathcal{S}}_{\psi_{\EE}}$ by 
    \begin{align}\label{eq character}
        \eta_{\EE}( \rho\otimes S_{a_i} \otimes S_{b_i}):= (-1)^{|Z(\EE)_{\rho,i}|+ \lfloor \half{b_i} \rfloor + l_i}\eta_i^{b_i}.
    \end{align}
\end{enumerate}
\end{defn}

The following theorem gives the recipe to compute $\langle \cdot, \pi(\EE)\rangle_{\psi_{\EE}}$.

\begin{thm}[{\cite[Theorem 3.6]{Ato20b}}]
    Suppose that $\pi(\EE) \neq 0$. Then $\langle \cdot, \pi(\EE)\rangle_{\psi_{\EE}}= \eta_{\EE}.$
\end{thm}
\begin{remark}\label{rmk character}\ 
\begin{enumerate}
    \item The definition of $Z(\EE)_{\rho,i}$ only depends on $\psi_{\EE}$ and the collection of admissible orders on each $I_{\rho}$.
    \item Suppose $\psi_{\EE}=\psi_{\EE'}$ and $\EE, \EE'$ have the same admissible order (Definition \ref{def multi-segment}(4)). Let $\ee_i, \ee_i'$ be the $i$-th extended $\Z$-segment in $\EE_{\rho}^{\Z}$ and $(\EE')_{\rho}^{\Z}$ respectively. An immediate consequence of \eqref{eq character} is that if $\ee_i$ is adjacent to $\ee_{i}'$, then
    \[ \eta_{\EE}(\rho\otimes S_{a_i}\otimes S_{b_i})= -\eta_{\EE'}(\rho\otimes S_{a_i}\otimes S_{b_i}).\]
    This is the main observation that leads us to the definitions for adjacent and interval in \S \ref{sec Z-extended segments}.
\end{enumerate}
\end{remark}

\subsection{Unitary induction of Arthur type}\label{sec unitary induction of Arthur type}
We recall the setting in \S \ref{sec-LIR}, and the description of the decomposition of 
\[ u_{\rho}(a,b) \rtimes \pi_0\]
for any representation $\pi_0$ of Arthur type and of good parity.

Let $\psi_0$ be a local Arthur parameter of good parity, and consider 
\[ \psi:= \psi_0 + (\rho\otimes S_{a} \otimes S_{b})^{\oplus 2}.\]
Recall that there is an injection between the component groups, which induces a surjection between their Pontryagin duals via restriction:
\[ \iota:\mathcal{S}_{\psi_0} \hookrightarrow   \mathcal{S}_{\psi},\ \ \ \iota^{\ast}:  \widehat{\mathcal{S}}_{\psi}  \twoheadrightarrow \widehat{\mathcal{S}}_{\psi_0}.\]
Note that the above maps are isomorphisms unless $\rho\otimes S_{a}\otimes S_{b}$ is of good parity and $\psi_0$ does not contain any summand isomorphic to $\rho\otimes S_{a}\otimes S_{b}$.

Fixing $\varepsilon_0 \in  \widehat{\mathcal{S}}_{\psi_0},$ it is known from \eqref{eq unitary induction Arthur} that
\begin{align*}
  \bigoplus_{\substack{\pi_{0} \in \Pi_{\psi_{0}},\\ \langle \cdot, \pi_{0} \rangle_{\psi}= \varepsilon_{0}}} u_{\rho}(a,b) \rtimes \pi_{0}= \bigoplus_{\substack{\pi \in \Pi_{\psi}, \\ \iota^{\ast} (\langle \cdot, \pi \rangle_{\psi})= \varepsilon_{0} }} \pi.
\end{align*}
 When $\psi_0$ is of good parity, every $\pi_0 \in \Pi_{\psi_0}$ is of the form $\pi_0= \pi(\EE)$ for some extended multi-segment $\EE$. Then we may describe the decomposition of the unitary induction $u_{\rho}(a,b) \rtimes \pi(\EE)$ explicitly in terms of the combinatorial data of $\EE$.

\begin{thm}[{\cite[Theorem 4.4]{Ato22d}}]\label{thm Atobe unitary induction}
    Suppose that $\pi(\EE) \neq 0$ and $\psi_{\EE}+ (\rho\otimes S_a \otimes S_b)^{\oplus 2}$ is of good parity. Let $A:=\half{a+b}-1$ and $B:=\half{a-b}.$    Then,
    \begin{align}\label{eq unitary induction Atobe}
        u_{\rho}(a,b) \rtimes \pi(\EE) = \bigoplus_{\ee \in \Eseg_{[A,B]_{\rho}}} \pi(\EE_{\ee}).
    \end{align}
   See Definition \ref{def transport extended multi-segments}(5) for the definition of $\EE_{\ee}$.
\end{thm}

\begin{remark}\label{rmk non-empty Atobe}
Note that $\EE_{\ee}$ satisfies ($P'$) by definition. By Theorem \ref{thm non-vanishing}, $\pi(\EE_{\ee})\neq 0$ if and only if $\ee \in \NV_{\EE}(S_{\EE})$, where $S_{\EE} \subseteq \Eseg_{[A,B]_{\rho}}$ is the interval given by the inequality \eqref{eq ast} for $\ee$ inside $\EE_{\ee}$ (see Remark \ref{rmk interval non-vanishing ineq}). A non-trivial implication of the above theorem is that under this setting, the interval $\NV_{\EE}(S_{\EE})$ is always non-empty.
\end{remark}

\section{Proof of the main theorem}\label{sec-proof}

In this section, we prove our main result Theorem \ref{thm A+,u}. First, we recall the setting as follows. 
Let $\pi \in \Pi_{A+}(G_n)$ and write
\begin{align}\label{eq decomp A+ pi}
    \pi= \bigtimes_{i \in I_{nu} } u_{\rho_i}(a_i, b_i) \lvert \cdot \rvert^{x_i} \rtimes \pi_A,
\end{align}
where $ 0 < x_i < \half{1}$ for any $i \in I_{nu}$ and $\pi_A \in \Pi_{A}(G_n)$. We define the subsets 
\[ \Sigma_{A+, \, u}(G_n) \subseteq \Sigma_{A+, \, H}(G_n) \subseteq \Pi_{A+}(G_n)\]
as follows.
\begin{defn}\  
    \begin{enumerate}
        \item We define  $\Sigma_{A+, \, H}(G_n)$ to be the set of $\pi\in \Pi_{A+}(G_n)$, given as in \eqref{eq decomp A+ pi}, such that  
        \[\bigtimes_{i \in I_{nu} } u_{\rho_i}(a_i, b_i) \lvert \cdot \rvert^{x_i} \times u_{\rho_i^{\vee}}(a_i, b_i) \lvert \cdot \rvert^{-x_i} \]
        is a unitary representation of a general linear group. Namely, the corresponding local Arthur parameter $\psi \in \Psi^{+}_{1/2}(G_n)$ lies in $\Psi^{+}_{\mathrm{unit}}(G_n)$.
        \item We define $\Sigma_{A+, \, u}(G_n)$ to be the set of $\pi\in \Pi_{A+}(G_n)$, given as in \eqref{eq decomp A+ pi}, such that $\pi \in \Sigma_{A+, \, H}(G_n)$ and for any $i \in I_{nu}$ such that $u_{\rho_i}(a_i,b_i)\rtimes \pi_A$ is reducible, the set 
\[ \{j \in I_{nu}\ | \ \rho_j \cong \rho_i, a_j=a_i,\ b_j=b_i\}\]
has even cardinality.
    \end{enumerate}
\end{defn}

\begin{remark}
An algorithm to determine whether $u_{\rho_i}(a_i,b_i)\rtimes \pi_A$ is reducible or not is stated in \cite[\S 5]{Ato22d}. We recall this algorithm now. By Theorem \ref{thm red from nu to gp}, we may write
    \[ \pi_A= \tau_{bp} \rtimes \pi_{gp},\]
    where $\pi_{gp}= \pi(\EE)$ is of good parity. Then $u_{\rho_i}(a_i,b_i)\rtimes \pi_A$ is reducible if and only if $u_{\rho_i}(a_i,b_i)\rtimes \pi(\EE)$ is reducible, which is equivalent to 
    \begin{enumerate}
        \item $ \psi_{\EE}+ (\rho_i \otimes S_{a_i} \otimes S_{b_i})^{\oplus 2}$ is of good parity, and
        \item $| \{\ee \in \Eseg_{[A_i,B_i]_{\rho_i}}\ | \ \pi(\EE_{\ee})\neq 0 \} | >1$, where $A_i= \half{a_i-b_i}-1$ and $B_i= \half{a_i+b_i}$,
    \end{enumerate}
    by Theorem \ref{thm Atobe unitary induction}.
\end{remark}
Here is the main result of this section.
\begin{thm}\label{thm main}
Let $\pi \in \Pi_{A+}(G_n)$. 
\begin{enumerate}
    \item The representation $\pi$ is Hermitian if and only if $\pi \in \Sigma_{A+, \, H}(G_n)$.
        \item The representation $\pi$ is unitary if and only if $\pi \in \Sigma_{A+, \, u}(G_n)$, that is,
        $$\Pi_{A+,\, u}(G_n)=\Sigma_{A+, \, u}(G_n).$$
        Hence, Theorem \ref{thm A+,u} holds. 
\end{enumerate}
\end{thm}

\subsection{\texorpdfstring{Tadi{\'c}'s result on weakly real representations}{}}
 We recall several notations and results from \cite{Tad09a}. Let 
 \begin{align*}
     \mathcal{C}_n&:= \{ \rho \in \Pi(\GL_n(F))\ | \ \rho \text{ is supercuspidal}\},\\
     \mathcal{C}_n^u&:= \{ \rho \in \Pi(\GL_n(F))\ | \ \rho \text{ is supercuspidal and unitary}\},\\
     \mathcal{C}_n^{sd}&:= \{ \rho \in \Pi(\GL_n(F))\ | \ \rho \text{ is supercuspidal and self-dual}\},
 \end{align*}
 and let 
\[ \mathcal{C}:= \bigsqcup_{n \geq 0} \mathcal{C}_n,\ \ \mathcal{C}^u:= \bigsqcup_{n \geq 0} \mathcal{C}_n^u,\ \ \mathcal{C}^{sd}:= \bigsqcup_{n \geq 0} \mathcal{C}_n^{sd}. \]
Note that $\mathcal{C} \supset \mathcal{C}^{u} \supset \mathcal{C}^{sd}.$ Also, there is a bijection
\begin{align*}
    \mathcal{C} & \to \mathcal{C}^{u} \times \R, \\
     \rho & \mapsto (\rho^u, x)
\end{align*}
such that $\rho \cong \rho^u \lvert \cdot \rvert^x$.
Throughout this section, we fix a subset $\mathcal{C}' \subseteq \mathcal{C}^{u}$ such that $\mathcal{C}' \cap (\mathcal{C}')^{\vee}= \emptyset$ 
and 
\[ \mathcal{C}^{u} \setminus \mathcal{C}^{sd}=  \mathcal{C}' \sqcup (\mathcal{C}')^{\vee},\]
where
\[ (\mathcal{C}')^{\vee}:= \{\rho \in \mathcal{C}\ | \ \rho^{\vee} \in \mathcal{C}'\}.\]

 Suppose $\pi\in \Pi(G_n)$ and 
\[ \pi \leq \rho_1 \lvert \cdot \rvert^{x_1} \times \cdots \times \rho_f \lvert \cdot \rvert^{x_f} \rtimes \pi_{sc},\]
where $\rho_i$'s are unitary supercuspidal representations of general linear groups, $x_i \in \R$ and $\pi_{sc}$ is supercuspidal. Then the representation $\pi$ is called \emph{weakly real} if $\rho_i \cong \rho_i^{\vee}$ for any $1 \leq i \leq f$. 

\begin{thm}[{\cite[Theorem 4.2]{Tad09a}}]\label{thm Tadic weakly real}
    Any irreducible representation $\pi$ can be uniquely written as an irreducible parabolic induction
\[ \pi = \theta(\pi) \rtimes X_{wr}(\pi), \]
where
\begin{itemize}
    \item $X_{wr}(\pi)$ is weakly real, and
    \item if we write 
    \[ \theta(\pi) \leq \rho_1' \lvert \cdot \rvert^{x_1'} \times \cdots \times \rho_k' \lvert \cdot \rvert^{x_k'}\]
    where $\rho_i'$ are unitary supercuspidal and $x_i' \in \R$, then $\rho_i' \in \mathcal{C}'$ for any $1 \leq i \leq k$.
\end{itemize}
 Moreover, $\pi$ is unitary (resp. Hermitian) if and only if both $\theta(\pi)$ and $X_{wr}(\pi)$ are unitary (resp. Hermitian).
\end{thm}
We remark that weakly real representations are always Hermitian. Thus, $\pi$ is Hermitian if and only if $\theta(\pi)$ is Hermitian.

We also recall the following well-known methods to verify unitarity. See \cite[\S 3]{Tad93} for a sketch of a proof for these facts.
\begin{lemma}\label{lem parabolic reduction}
Let  $\tau \in \Pi(\GL_d(F))$ and $\sigma \in \Pi(G_n)$.
\begin{enumerate}
    \item [(a)] (Complementary series) Suppose that $\Pi_x:=\tau\lvert \cdot \rvert^{x} \rtimes \sigma$ is irreducible and Hermitian for any $x$ lying in a connected set $\Sigma \subset \R$. Then $\Pi_x$ is unitary for some $x \in \Sigma$ if and only if $\Pi_x$ is unitary for any $x \in \Sigma$.
    \item [(b)] (Limit of complementary series) In the setting of Part (a), if $\Pi_x$ is unitary for some $x \in \Sigma$, then for any $y$ lying in the closure of $\Sigma$,
        any irreducible subquotient of $\Pi_{y}$ is unitary.
        \item [(c)](Unitary parabolic reduction) Suppose that $\tau$ and $\sigma$ are both Hermitian and that $\tau\rtimes \sigma$ is irreducible. Then $\tau \rtimes \sigma$ is unitary if and only if $\tau$ and $\sigma$ are both unitary.
\end{enumerate}
    
\end{lemma}

\subsection{Decomposition of unitary inductions} 

In this subsection, we obtain results on the reducibility of certain unitary parabolic inductions.

Recall the setting in \S \ref{sec unitary induction of Arthur type} that we have a local Arthur parameter $\psi_0$ of good parity and
\[ \psi:= \psi_0 + (\rho\otimes S_{a} \otimes S_{b})^{\oplus 2}\]
is also of good parity. There is a surjection $\iota^{\ast}:  \widehat{\mathcal{S}}_{\psi}  \twoheadrightarrow \widehat{\mathcal{S}}_{\psi_0}$ given by restriction. For any fixed $\pi_0 \in \Pi_{\psi_0}$, we combine the results in the two previous sections to study the set of characters
\[ (\iota^{\ast})^{-1}(\pi_0):=\{ \langle \cdot, \pi \rangle_{\psi}\ | \ \pi \leq u_{\rho}(a,b) \rtimes \pi_0  \}.\]

Write $\varepsilon_0:=\langle \cdot, \pi_0 \rangle_{\psi_0}$ for short. According to \eqref{eq unitary induction Arthur}, the above set is contained in 
\[ (\iota^{\ast})^{-1} (\varepsilon_0)=\{ \varepsilon \in \widehat{\mathcal{S}}_{\psi} \ | \ \varepsilon|_{{\mathcal{S}}_{\psi_0}}  = \langle \cdot , \pi_0 \rangle\}. \]
We first show that $|(\iota^{\ast})^{-1} (\varepsilon_0)|\in \{1,2\}$ and write down the member(s) explicitly. Indeed, write
\[\psi_0= \bigoplus_{i \in I_0} \rho_i \otimes S_{a_i}\otimes S_{b_i},\ \psi= \bigoplus_{i \in I} \rho_i \otimes S_{a_i}\otimes S_{b_i}, \]
where we identify $I= I_0 \sqcup \{ i_0, i_0' \}$. Thus, 
\[ \rho_{i_0} \otimes S_{a_{i_0}}\otimes S_{b_{i_0}}= \rho\otimes S_{a}\otimes S_b =\rho_{i_0'} \otimes S_{a_{i_0'}}\otimes S_{b_{i_0'}}.\]
Recall from \S \ref{sec character} that we have identified $\widehat{\mathcal{S}}_{\psi_0}$ (resp. $\widehat{\mathcal{S}}_{\psi}$) with a subset of functions 
\[\Sigma_0\subseteq \{\chi_0: I_0 \to \{\pm 1\}\}\ \  \text{(resp. } \Sigma \subseteq \{\chi: I \to \{\pm 1\}\})\] with several requirements. Under this identification, the surjection $\iota^{\ast}$ is exactly the restriction of the corresponding functions
\[ \iota^{\ast}: \chi \mapsto \chi|_{I_0}.\]
Fixing $\varepsilon_0 \in \widehat{\mathcal{S}}_{\psi_0}$, we define functions
\begin{align*}
    \varepsilon_0^{+} (i):= \begin{cases}
        \varepsilon_0(i) & \text{if }i \in I_0,\\
        1 &\text{if }i \in \{i_0, i_0'\},
    \end{cases}\ \ \ \     \varepsilon_0^{-} (i):= \begin{cases}
        \varepsilon_0(i) & \text{if }i \in I_0,\\
        -1 &\text{if }i \in \{i_0, i_0'\}.
    \end{cases} 
\end{align*}
Note that the functions $\varepsilon_0^{+} (i)$ and $\varepsilon_0^{-} (i)$ both lie in $\Sigma$ if and only if $\rho_i \otimes S_{a_i}\otimes S_{b_i} \not\cong \rho \otimes S_a \otimes S_b $ for any $i \in I_{\rho}$. Therefore, we conclude that 
\begin{align*}
    (\iota^{\ast})^{-1} (\varepsilon_0)= \begin{cases}
\{\varepsilon_0^{ \varepsilon_0(i)}\} & \text{if there exists an }i \in I_0 \text{ such that }\rho_i \otimes S_{a_i}\otimes S_{b_i} \cong \rho \otimes S_a \otimes S_b,\\
\{\varepsilon_0^+, \varepsilon_0^{-} \} & \text{otherwise.}
    \end{cases}
\end{align*}
For simplicity, we will write $\psi_0 \supseteq \rho \otimes S_a \otimes S_b$ if there exists an $i \in I_0$ such that $\rho_i \otimes S_{a_i}\otimes S_{b_i} \cong \rho \otimes S_a \otimes S_b$, and write $ \psi_0 \not\supseteq \rho \otimes S_a \otimes S_b$ otherwise.

Recall that we have defined an element $s_u \in \mathcal{S}_{\psi}$ corresponding to the summand $\rho \otimes S_{a}\otimes S_{b}$ inside $\psi$ in \S \ref{subsec LIR}. Suppose that $\pi \leq u_{\rho}(a,b)\rtimes \pi_0$ and $\langle \cdot, \pi \rangle_{\psi}= \varepsilon_0^{\pm}$. Then
\[ \langle s_u , \pi\rangle_{\psi}= \pm 1.\]
Thus, the following theorem implies Theorem \ref{thm base case reducible components}.

\begin{thm}\label{thm reducibility unitary induction}
    Keep the notation above. Suppose that $ \psi_0 \not\supseteq \rho \otimes S_a \otimes S_b$.    Let
\[ m^{\pm}:=| \{\pi \leq u_{\rho}(a,b)\rtimes \pi_0 \ | \ \langle \cdot, \pi\rangle = \varepsilon_0^{\pm}\}|. \]
    Then $|m^+- m^-| \leq 1$. In particular, if $u_{\rho}(a,b)\rtimes \pi_0$ reduces, then $(\iota^{\ast})^{-1} (\pi_0)= \{\varepsilon_0^{+}, \varepsilon_0^{-}\}$.
\end{thm}
\begin{proof}
Write $\pi_0=\pi(\EE)$ with $\psi_{\EE}= \psi_0$. According to \eqref{eq unitary induction Atobe}, we have
    \[  u_{\rho}(a,b) \rtimes \pi(\EE) = \bigoplus_{\ee \in \Eseg_{[A,B]_{\rho}}} \pi(\EE_{\ee})= \bigoplus_{j\in S} \pi(\EE_{\ee_j}),  \]
    where $S:= \{ \ee \in \Eseg_{[A,B]_{\rho}}  \ | \ \pi(\EE_{\ee})\neq 0 \} = \NV_{\EE}(S_{\EE})$, see Remark \ref{rmk non-empty Atobe}. Now Proposition \ref{prop NV segment}(1) implies that $S$ is an interval. Thus, we write $S= \{\ee_1, \ldots, \ee_r \}$ where $\ee_i$ is adjacent to $\ee_{i+1}$. Then Remark \ref{rmk character}(2) implies that 
    \[ \eta_{\EE_{\ee_i}}(\rho \otimes S_a \otimes S_b)= - \eta_{\EE_{\ee_{i+1}}}(\rho \otimes S_a \otimes S_b).\]
    Let $\epsilon:= \eta_{\EE_{\ee_1}}(\rho \otimes S_a \otimes S_b)$. We conclude that 
    \begin{align*}
         \{\pi \leq u_{\rho}(a,b)\rtimes \pi_0 \ | \ \langle \cdot, \pi\rangle = \varepsilon_0^{\epsilon}\}= \{\pi(\EE_{\ee_j})\ | \ 1 \leq j \leq r,\ 1 \equiv j \mod 2 \},\\
         \{\pi \leq u_{\rho}(a,b)\rtimes \pi_0 \ | \ \langle \cdot, \pi\rangle = \varepsilon_0^{-\epsilon}\}= \{\pi(\EE_{\ee_j})\ | \ 1 \leq j \leq r,\ 0 \equiv j \mod 2 \}.
    \end{align*}
    This completes the proof of the theorem. 
\end{proof}

As a corollary, we prove an irreducibility result.

\begin{cor}\label{cor irreducible}
    Suppose that $ \rho_i \otimes S_{a_i}\otimes S_{b_i}$, $i=1,\ldots, r$, are of good parity.  Let $\pi= \pi(\EE)\neq 0$.    If $u_{\rho_i}(a_i, b_i) \rtimes \pi(\EE)$ is irreducible for any $i=1 ,\ldots, r$,  then
    \[  \bigtimes_{i=1}^r u_{\rho_i}(a_i, b_i) \rtimes \pi(\EE)\]
    is also irreducible.
\end{cor}
\begin{proof}
Let $\psi_0:= \psi_{\EE}$. For $1 \leq k \leq r$, let
\[ \Pi_k:= \bigtimes_{i=1}^k  u_{\rho_i}(a_i, b_i) \rtimes \pi(\EE).\]
Suppose the contrary that $\Pi_{k}$ is reducible and $\Pi_{k-1}$ is irreducible for some $2 \leq k \leq r$. Let 
\begin{align*}
     \psi& := \psi_0+ \bigoplus_{i=1}^{k} (\rho_i \otimes S_{a_i} \otimes S_{b_i})^{\oplus 2},\\
    \psi'& := \psi_0+ \bigoplus_{i=1}^{k-1} (\rho_i \otimes S_{a_i} \otimes S_{b_i})^{\oplus 2}, \\
    \psi_k &:= \psi_0+(\rho_k \otimes S_{a_k} \otimes S_{b_k})^{\oplus 2}.
\end{align*}
Let $\varepsilon':= \langle \cdot, \Pi_{k-1} \rangle_{\psi'} \in \widehat{\mathcal{S}}_{\psi'}$. Then Theorem \ref{thm reducibility unitary induction} implies that there exist $ \Pi_{k-1}^{+}, \Pi_{k-1}^{-} \leq \Pi_k$ such that 
\[  \langle \cdot , \Pi_{k-1}^{\pm} \rangle_{\psi}( \rho_k \otimes S_{a_k} \otimes S_{b_k})= \pm 1. \]
On the other hand, $\pi_k:=  u_{\rho_k}(a_k, b_k) \rtimes \pi(\EE)$ is irreducible by assumption. Then by repeatedly applying \eqref{eq unitary induction Arthur}, for any irreducible subquotient $\sigma$ of
\[  \Pi_k= \bigtimes_{i=1}^{k-1}  u_{\rho_i}(a_i, b_i) \rtimes \pi_k, \]
its character value at $\rho_k \otimes S_{a_k}\otimes S_{b_k}$  must be
\[  \langle \cdot, \sigma\rangle_{\psi} (\rho_k \otimes S_{a_k}\otimes S_{b_k} )= \langle \cdot, \sigma\rangle_{\psi}|_{\mathcal{S}_{\psi_k}} (\rho_k \otimes S_{a_k}\otimes S_{b_k} )= \langle \cdot, \pi_k\rangle_{\psi_k}(\rho_k \otimes S_{a_k}\otimes S_{b_k} ).   \]
This gives a contradiction to the existence of $\Pi_{k-1}^{+}$ or $\Pi_{k-1}^{+}$ and completes the proof of the corollary.
\end{proof}

Finally, we prove the following theorem (which is obtained independently in \cite[Corollary 3.33]{BS25} at the same time), which is crucial in the reduction process.

\begin{thm}\label{thm reduction}
Suppose that $\rho_1 \otimes S_{a_1}\otimes S_{b_1} \not\cong \rho_2 \otimes S_{a_2}\otimes S_{b_2}$
and that $ \psi_0+ (\rho_i \otimes S_{a_i}\otimes S_{b_i})^{\oplus 2}$, $i=1,2$, are of good parity. Fix a $\pi_0 \in \Pi_{\psi_0}$ and let
\[ \Pi_{i}:=u_{\rho_i}(a_i, b_i)\rtimes \pi_0. \]
If $\Pi_1,\Pi_2$ are both reducible, then after relabeling if necessary, there exists an irreducible subquotient $\pi_1 \leq \Pi_1$ such that
\[ u_{\rho_2}(a_2,b_2)\rtimes \pi_1\]
is reducible.
\end{thm}
\begin{proof}
    If $\rho_1 \not\cong \rho_2$, then for any $\pi_1 \leq \Pi_1$, we have
    \[| \{ \pi \leq  u_{\rho_2}(a_2,b_2)\rtimes \pi_1\} |= | \{\pi\leq u_{\rho_2}(a_2,b_2)\rtimes \pi_0\}|>1 \]
    by the Jantzen decomposition (\cite[Theorem 9.3(6)]{Jan97}). Thus we assume that $\rho_1 \cong \rho_2 \cong \rho$ in the rest of the proof. Let $A_i:= \half{a_i+b_i}-1$ and $B_i:= \half{a_i-b_i}$ and let $\Delta_i:= [A_i,B_i]_{\rho}$. We assume that 
    \begin{align}
        \Delta_1^{\Z}:=[ \lfloor A_1\rfloor , \lfloor B_1\rfloor] \preceq [ \lfloor A_2\rfloor , \lfloor B_2\rfloor]=: \Delta_2^{\Z}
    \end{align} (see Definition \ref{def non-vanishing}(3) for the definition of $\preceq$) by relabeling them if necessary. 

    Write $\pi_0=\pi(\EE)$ with $\psi_{\EE}= \psi_0$.  Suppose that there exists an $\pi_1= \pi(\EE_{\ee_1}) \leq \Pi_1$ such that $u_{\rho}(a_2, b_2) \rtimes \pi_1$ is irreducible. By Remark \ref{rmk non-empty Atobe}, we write
    \[ S_i:= \{ \ee \in \Eseg_{[A_i,B_i]_{\rho}}  \ | \ \pi(\EE_{\ee})\neq 0 \}= \NV_{\EE}(S_{i,\EE}), \]
    which are intervals by Proposition \ref{prop NV segment}(1). Note that $|S_i|= |\{\pi \leq \Pi_i\}|>1$ by assumption. Thus, we can take an $\ee_1' \in S_1$ that is adjacent to $\ee_1$ and let $\pi_1':= \pi(\EE_{\ee_1'})$. We claim that $u_{\rho}(a_2,b_2)\rtimes \pi_1'$ reduces, or equivalently, $|\NV_{\EE_{\ee_1'}}(S_{2,\EE})| >1.$
    By the above discussion and Remark \ref{rmk non-empty Atobe}, we have
    \begin{itemize}
        \item $ \Delta_1^{\Z} \preceq \Delta_2^{\Z} $ and $\Delta_1 \neq \Delta_2$.
        \item $|\NV_{\EE}(S_{1,\EE})| >1$ and $|\NV_{\EE}(S_{2,\EE})| >1$.
        \item $\ee_1, \ee_1' \in \NV_{\EE}(S_{1,\EE})$ are adjacent, $|\NV_{\EE_{\ee_1}}(S_{2,\EE})|=1 $ and $|\NV_{\EE_{\ee_1'}}(S_{2,\EE})|>0 .$ 
    \end{itemize}
    Thus, Proposition \ref{prop NV segment}(2) implies that $|\NV_{\EE_{\ee_1'}}(S_{2,\EE})|=2$. This verifies the claim and completes the proof of the theorem.   
\end{proof}

\subsection{\texorpdfstring{Proof of Theorem \ref{thm main}}{}}
We begin with some notation. Let $\pi\in\Pi_{A+}(G_n)$ be written as in \eqref{eq decomp A+ pi}. Decompose $I_{nu}= I_{nu,sd} \sqcup I_{nu,nsd}$ where
\[ I_{nu,sd}= \{i \in I_{nu}\ | \ \rho_i \cong \rho_i^{\vee}\}.\]
We set
\[ \tau_{nu,nsd}:= \bigtimes_{i \in I_{nu,nsd} } u_{\rho_i}(a_i, b_i) \lvert \cdot \rvert^{x_i},\ \tau_{nu,sd}:=\bigtimes_{i \in I_{nu,sd} } u_{\rho_i}(a_i, b_i) \lvert \cdot \rvert^{x_i} .\]

For $i \in I_{nu,nsd}$, let
\[ (\rho_i', x_i'):= \begin{cases}
    (\rho_i, x_i) & \text{ if }\rho_i \in \mathcal{C}',\\
    (\rho_i^{\vee},-x_i) &\text{ otherwise,}
\end{cases}\]
and let $I_{nu,nsd}^{+}:= \{i \in I_{nu,nsd}\ | \ x_i' >0\}.$ Note that $\rho_i'\in\mathcal{C}'$ for any $i\in I_{nu,nsd}.$ Thus we may rewrite \eqref{eq decomp A+ pi} as 
\[ \pi= \left(\bigtimes_{i \in I_{nu,nsd} } u_{\rho_i'}(a_i, b_i) \lvert \cdot \rvert^{x_i'} \times \theta(\pi_A) \right) \rtimes  \left( \bigtimes_{i \in I_{nu,sd} } u_{\rho_i}(a_i, b_i) \lvert \cdot \rvert^{x_i} \rtimes X_{wr}(\pi_A)\right),  \]
which gives the decomposition $\pi=\theta(\pi)\rtimes X_{wr}(\pi)$. 
 Since $\theta(\pi_A)$ is unitary, it is Hermitian. The Hermitian contragredient of $\theta(\pi) $ is given by 
\[\theta(\pi)^+:= \bigtimes_{i \in I_{nu,nsd} } u_{\rho_i'}(a_i, b_i) \lvert \cdot \rvert^{-x_i'} \times \theta(\pi_A). \]

Now we show Part (1) of Theorem \ref{thm main}. Suppose that $\pi$ is Hermitian. Then Theorem \ref{thm Tadic weakly real} implies that  $\theta(\pi)$ is also Hermitian, i.e.  $\theta(\pi) \cong \theta(\pi)^+$. By comparing the $L$-parameters of $\theta(\pi)$ and $\theta(\pi)^+$, we must have $|I_{nu,nsd}|=2|I_{nu,nsd}^+|$ and
\[ \bigtimes_{i \in I_{nu,nsd} } u_{\rho_i'}(a_i, b_i) \lvert \cdot \rvert^{x_i'}=  \bigtimes_{i \in I_{nu,nsd}^+ } u_{\rho_i'}(a_i, b_i) \lvert \cdot \rvert^{x_i'}\times u_{\rho_i'}(a_i, b_i) \lvert \cdot \rvert^{-x_i'}.  \]
Therefore,
\begin{align}\label{eq nu,nsd^+}
  \tau_{nu,nsd}=\bigtimes_{i \in I_{nu,nsd} } u_{\rho_i}(a_i, b_i) \lvert \cdot \rvert^{x_i}= \bigtimes_{i \in I_{nu,nsd}^+ } u_{\rho_i}(a_i, b_i) \lvert \cdot \rvert^{x_i}\times u_{\rho_i^{\vee}}(a_i, b_i) \lvert \cdot \rvert^{x_i}.
\end{align} 
By the classification of the unitary dual of general linear groups in \cite{Tad86},  $\tau_{nu,nsd} \times \tau_{nu,nsd}^{\vee}$ is unitary. This implies that $ \pi \in \Sigma_{A+, \, H}(G_n)$.

Conversely, suppose that $ \pi \in \Sigma_{A+, \, H}(G_n)$, which implies that
\[ \bigtimes_{i \in I_{nu,nsd} } u_{\rho_i}(a_i, b_i) \lvert \cdot \rvert^{x_i} \times  u_{\rho_i^{\vee}}(a_i, b_i) \lvert \cdot \rvert^{-x_i} \]
is unitary. By the classification of the unitary dual of general linear groups, there exists an involution $\iota: I_{nu,nsd} \to I_{nu,nsd}$ such that $ \rho_{\iota(i)}\cong \rho_i^{\vee}$ and $(a_{\iota(i)},b_{\iota(i)}, x_{\iota(i)})= (a_i, b_i, x_i).$ Since $ \mathcal{C}' \cap (\mathcal{C}')^{\vee}= \emptyset$, we see that $\iota(I_{nu,nsd}^+) \cap I_{nu,nsd}^+= \emptyset.$ Then it is not hard to check that $\theta(\pi)^+ \cong \theta(\pi)$; hence $\pi$ is Hermitian. This completes the proof of Part (1) of Theorem \ref{thm main}.

Now we prove Part (2). We may assume that $\pi \in \Sigma_{A+, \, H}(G_n)$, otherwise $\pi$ cannot be unitary. We first reduce to the weakly real case. By \eqref{eq nu,nsd^+}, we have
\begin{align*}
    \pi&= \tau_{nu,nsd} \times \tau_{nu,sd} \rtimes \pi_A\\
    &= \left(\bigtimes_{i \in I_{nu,nsd}^+ } u_{\rho_i}(a_i, b_i) \lvert \cdot \rvert^{x_i}\times u_{\rho_i^{\vee}}(a_i, b_i) \lvert \cdot \rvert^{x_i}\right ) \rtimes (\tau_{nu,sd} \times \theta(\pi_A) \rtimes X_{wr}(\pi_A))\\
    &= \left(\bigtimes_{i \in I_{nu,nsd}^+ } u_{\rho_i}(a_i, b_i) \lvert \cdot \rvert^{x_i}\times u_{\rho_i}(a_i, b_i) \lvert \cdot \rvert^{-x_i}\right)\times \theta(\pi_A) \rtimes (\tau_{nu,sd} \rtimes X_{wr}(\pi_A)).
\end{align*}
Since the first two terms are unitary by \cite{Tad86}, Lemma \ref{lem parabolic reduction}(c) implies that $\pi$ is unitary if and only if $\tau_{nu,sd} \rtimes X_{wr}(\pi_A)$ is unitary. Also, Theorem \ref{thm red from nu to gp} implies that $u_{\rho_i}(a_i,b_i) \rtimes \pi_{A}$ is irreducible for any $i \in I_{I_{nu,nsd}}$. It follows from the definition that $ \pi \in \Sigma_{A+, \, u}(G_n)$ if and only if $\tau_{nu,sd} \rtimes X_{wr}(\pi_A) \in \Sigma_{A+, \, u}(G_m)$ for some $m \leq n$. Therefore, we assume that $I_{nu,nsd}$ is empty and $\pi_A$ is weakly real in the rest of the proof.

Next, we reduce to the good parity case. Write $\pi_A= \tau_{bp} \rtimes \pi_{A,gp}$ by Theorem \ref{thm red from nu to gp}, where $\pi_{A,gp}$ is of good parity. Decompose $I_{nu,sd}= I_{nu,gp} \sqcup I_{nu,bp}$, where 
\[ I_{nu,gp}= \{i \in I_{nu,sd}\ | \ \rho_i \otimes S_{a_i} \otimes S_{b_i} \text{ is of good parity} \}.\]
Write $I_{nu,bp}=\{1,\ldots, r\}$. Theorem \ref{thm red from nu to gp} implies that 
\[ \bigtimes_{i=1}^r u_{\rho_i}(a_i,b_i)\lvert \cdot \rvert^{y_i} \rtimes \pi_A  \]
is irreducible for any $ 0 \leq y_i <\half{1}, i=1,\ldots, r$. Using the Jantzen decomposition (\cite[Theorem 9.3(6)]{Jan97}), we also obtain the irreducibility of 
\[ \Pi_{y_1,\ldots, y_r}= \bigtimes_{i=1}^r u_{\rho_i}(a_i,b_i)\lvert \cdot \rvert^{y_i} \rtimes \left(\bigtimes_{i \in I_{nu,gp}} u_{\rho_i}(a_i,b_i)\lvert \cdot \rvert^{x_i} \rtimes \pi_A\right).  \]
Note that $\Pi_{y_1,\ldots, y_r}$ is also Hermitian since it is weakly real. Therefore, Lemma \ref{lem parabolic reduction}(a) implies that $\pi= \Pi_{x_1,\ldots, x_r}$ is unitary if and only if 
\begin{align*}
     \Pi_{0,\ldots, 0}&= \bigtimes_{i=1}^r u_{\rho_i}(a_i,b_i) \rtimes \left(\bigtimes_{i \in I_{nu,gp}} u_{\rho_i}(a_i,b_i)\lvert \cdot \rvert^{x_i} \times \tau_{bp} \rtimes \pi_{A,gp} \right)\\
     &=\left( \bigtimes_{i=1}^r u_{\rho_i}(a_i,b_i)\right) \times \tau_{bp} \rtimes \left(\bigtimes_{i \in I_{nu,gp}} u_{\rho_i}(a_i,b_i)\lvert \cdot \rvert^{x_i} \rtimes \pi_{A,gp} \right)
\end{align*}
is unitary. Since the first two terms of the last line above are both unitary, Lemma \ref{lem parabolic reduction}(c) implies that $\Pi_{0,\ldots,0}$ is unitary if and only if  $\pi_{gp}:=\bigtimes_{i \in I_{nu,gp}} u_{\rho_i}(a_i,b_i)\lvert \cdot \rvert^{x_i} \rtimes \pi_{A,gp} $  is unitary. Again, Theorem \ref{thm red from nu to gp} implies that $\pi \in \Sigma_{A+, \, u}(G_n)$ if and only if $\pi_{gp} \in \Sigma_{A+, \, u}(G_m)$ for some $m\leq n$. Therefore, we assume that $I_{nu, bp}$ is empty and $\pi_A$ is of good parity in the rest of the proof.

Now we write $\pi_A=\pi(\EE)$. By Lemma \ref{lem parabolic reduction}(a), we may assume that $x_i = x_j$ if $ u_{\rho_i}(a_i,b_i) \cong u_{\rho_j}(a_j,b_j)$. Then we rewrite 
\begin{align*}
     \pi= \bigtimes_{i=1}^r (u_{\rho_i}(a_i,b_i)\lvert \cdot \rvert^{x_i})^{\times m_i} \rtimes \pi(\EE), 
\end{align*}
where $m_i$ denotes the multiplicity and $u_{\rho_i}(a_i,b_i) \not\cong u_{\rho_j}(a_j,b_j)$ for any $1 \leq i \neq j \leq r$. Assume that 
\[m_i = \begin{cases}
    2 m_i' +1 &\text{ if }1 \leq i \leq s,\\
    2m_i' &\text{ if } s< i\leq r,
\end{cases}\]
where $m_i' \in \Z_{\geq 0}$. Then 
\begin{align*}
    \pi& = \bigtimes_{i=1}^r ( u_{\rho_i}(a_i,b_i)\lvert \cdot \rvert^{x_i} \times u_{\rho_i}(a_i,b_i)\lvert \cdot \rvert^{x_i}   )^{\times m_i'} \rtimes  \left(\bigtimes_{i=1}^s ( u_{\rho_i}(a_i,b_i)\lvert \cdot \rvert^{x_i} \rtimes \pi(\EE) \right)\\
    &= \bigtimes_{i=1}^r ( u_{\rho_i}(a_i,b_i)\lvert \cdot \rvert^{x_i} \times u_{\rho_i}(a_i,b_i)\lvert \cdot \rvert^{-x_i}   )^{\times m_i'} \rtimes  \left(\bigtimes_{i=1}^s ( u_{\rho_i}(a_i,b_i)\lvert \cdot \rvert^{x_i} \rtimes \pi(\EE) \right).
\end{align*}
Since each factor in the first product is unitary, Lemma \ref{lem parabolic reduction}(c) implies that $\pi$ is unitary if and only if $\pi_{mf}:=\bigtimes_{i=1}^s ( u_{\rho_i}(a_i,b_i)\lvert \cdot \rvert^{x_i} \rtimes \pi(\EE) $ is unitary. Clearly, $\pi \in \Sigma_{A+, \, u}(G_n)$ if and only if $\pi_{mf} \in \Sigma_{A+, \, u}(G_m)$ for some $m\leq n$. Therefore, we assume that 
\[ \pi= \bigtimes_{i=1}^r u_{\rho_i}(a_i,b_i)\lvert \cdot \rvert^{x_i} \rtimes \pi(\EE), \]
where $(\rho_i,a_i,b_i)$ are pairwise distinct
in the rest of the proof.

Now suppose that $\pi \in \Sigma_{A+, \, u}(G_n)$. That is, $u_{\rho_i}(a_i, b_i) \rtimes \pi(\EE)$ is irreducible for any $1 \leq i \leq r$. Then Corollary \ref{cor irreducible} implies that 
\[ \Pi_{y_1,\ldots, y_r}:= \bigtimes_{i=1}^r u_{\rho_i}(a_i,b_i)\lvert \cdot \rvert^{y_i} \rtimes \pi(\EE) \]
is irreducible for any $\{0 \leq y_i < \half{1}\}_{i=1,\ldots, r}$. Thus, Lemma \ref{lem parabolic reduction}(a) implies that $\pi=\Pi_{x_1,\ldots, x_r}$ is unitary.

Finally, suppose that $\pi \not\in \Sigma_{A+, \, u}(G_n)$. After rearrangement, we assume that for some $s \geq  1$,
\begin{align*}
    \text{length}( u_{\rho_i}(a_i, b_i) \rtimes \pi(\EE))
         &> 1 \ \text{ if }1 \leq i \leq s \ \mathrm{and} \\
         \text{length}( u_{\rho_i}(a_i, b_i) \rtimes \pi(\EE))&=1 \ \text{ if }s+1 \leq i \leq r.
\end{align*}
We apply induction on $s$ to show that $\pi$ is not unitary. Let 
\[ \pi(\EE') := \bigtimes_{i=s+1}^r u_{\rho_i}(a_i,b_i) \rtimes \pi(\EE). \]
Note that this is irreducible by Corollary \ref{cor irreducible}.
By Lemma \ref{lem parabolic reduction}(a), $\pi$ is not unitary if and only if 
\[ \pi':= \bigtimes_{i=1}^s u_{\rho_i}(a_i,b_i)\lvert \cdot \rvert^{x_i} \rtimes \pi(\EE') \]
is not unitary. Moreover, for any $ 1 \leq j \leq s$, the parabolic induction $u_{\rho_j}(a_j,b_j)\rtimes \pi(\EE')$ is still reducible since
\begin{align*}
    \text{length}(u_{\rho_j}(a_j,b_j)\rtimes \pi(\EE') ) &=\text{length}( \bigtimes_{i=s+1}^r u_{\rho_i}(a_i,b_i)\lvert \cdot \rvert^{x_i}  \times u_{\rho_j}(a_j,b_j)\rtimes \pi(\EE) )\\
    &\geq \text{length}(  u_{\rho_j}(a_j,b_j)\rtimes \pi(\EE))\\
    &\geq 2.
\end{align*}

If $s=1$, the non-unitarity of 
\[u_{\rho_1}(a_1,b_1)\lvert \cdot \rvert^{x_1'}\rtimes \pi(\EE')\]
follows from Corollary \ref{cor base case non-unitary}. Suppose $s > 1$. By Theorem \ref{thm reduction}, we may assume that there exists an irreducible subquotient $\pi_s$ of $u_{\rho_s}(a_s,b_s) \rtimes \pi(\EE')$ such that $u_{\rho_1}(a_1, b_1) \rtimes \pi_s$ is still reducible. Now Lemma \ref{lem parabolic reduction}(b) implies that $\pi'$ is not unitary if 
\[ \pi'':= \bigtimes_{i=1}^{s-1} u_{\rho_i}(a_i,b_i)\lvert \cdot \rvert^{x_i} \rtimes \pi_s \]
is not unitary. Note that $\pi'' \not\in \Sigma_{A+, \, u}(G_n)$ since $u_{\rho_1}(a_1, b_1) \rtimes \pi_s$ reduces by our choice of $\pi_s$.  Therefore, the induction hypothesis guarantees that $\pi''$ is not unitary. We conclude that $\pi$ is also not unitary. This completes the proof of the theorem.
\qed

\section{\texorpdfstring{Proof of Theorem \ref{thm Ramanujan 1 intro}}{}}\label{sec-Piv}

In this section, we prove Theorem \ref{thm Ramanujan 1 intro} which we recall now.
Let $\Pi$ be a self-dual irreducible cuspidal automorphic representation of $\GL_N(\A_k)$. Let $v$ be a finite place of $k$. By the classification of the unitary dual of $\GL_N(k_v)$ in \cite{Tad86}, we can write
\begin{align}\label{eq Piv sec 7}
    \Pi_v\cong u_{\rho_1}(a_1,1)\lvert \cdot \rvert^{x_1} \times \cdots \times u_{\rho_s}(a_s,1) \lvert \cdot \rvert^{x_s} 
    \times \pi_{t}
    \times u_{\rho_s}(a_s,1) \lvert \cdot \rvert^{-x_s} \times \cdots \times u_{\rho_1}(a_1,1) \lvert \cdot \rvert^{-x_1},
\end{align}
where $\rho_i$ are unitary supercuspidal representations, $0<x_i<\half{1}$, and $\pi_t$ is a tempered generic representation of the form
\[ \pi_t= \bigtimes_{i=s+1}^r u_{\rho_i}(a_i,1). \]
We say $u_{\rho_i}(a_i,1)$ is of orthogonal type (resp. symplectic type) if
\begin{itemize}
    \item $\rho_i$ is of orthogonal type and $a_i$ is odd (resp. even), or
    \item $\rho_i$ is of symplectic type and $a_i$ is even (resp. odd).
\end{itemize}
We say $u_{\rho}(a,1)$ is a factor of $\pi_t$ if there exists a $s+1 \leq i\leq r $ such that $\rho_i \cong \rho$ and $a_i=a$.

\begin{thm}[{Theorem \ref{thm Ramanujan 1 intro}}]\label{thm Ramanujan}
Continue the above setting. Suppose that $\Pi$ is of orthogonal type (resp. symplectic type). For any $1 \leq i \leq s$, if $u_{\rho_i}(a_i,1)$ is of orthogonal type (resp. symplectic type) and $u_{\rho_i}(a_i,1)$ is not a factor of $\pi_t$, then the set 
\[\{ 1 \leq j \leq s\ | \ \rho_j \cong \rho_i \text{ and }a_j=a_i\}\]
        has even cardinality. 
\end{thm}
\begin{proof}
We first prove the case that $\Pi$ is of orthogonal type. Consider the global Arthur parameter $\psi= \Pi \boxtimes S_2 $ of $ \SO_{2N+1}(\A_k)$. The corresponding local Arthur parameter at $v$ is
 \[ \psi_v= \bigoplus_{i=1}^s (\rho_i \lvert \cdot \rvert^{x_i} \otimes S_{a_i} \otimes S_{2}+ \rho_i \lvert \cdot \rvert^{-x_i} \otimes S_{a_i} \otimes S_{2}) + (\psi_{t})_v\]
 where
  \[ (\psi_t)_v= \bigoplus_{i=s+1}^r \rho_i \otimes S_{a_i}\otimes S_2\]
is a local Arthur parameter of some $\SO_{2m+1}(k_v)$. For simplicity of the notation, we assume that $\rho_i \otimes S_{a_i}\otimes S_2$ is of orthogonal type for any $s+1 \leq i \leq r$. The general case follows easily by Theorem \ref{thm red from nu to gp}.

Regard $I= \{s+1, \ldots, r\}$ as a total ordered set. Relabeling if necessary, we may assume that $a_i$ is non-decreasing. For each $\rho$, define $I_{\rho}:= \{i \in I\ | \ \rho_i \cong \rho\}$ with the total order inherited from $I$. Define
\[ \EE:= \bigcup_{\rho} \{ ([\half{a_i}, \half{a_i-2}]_{\rho}, 1,1) \}_{i \in I_{\rho}}. \]
Then it is clear that $\pi(\EE)$ is nonzero, and it is the only representation in $\Pi_{(\psi_{t})_v}$ such that $\langle \cdot, \pi(\EE)\rangle_{\psi}$ is the trivial character. 

Now {for a fixed $i\in\{1,\dots,s\},$} suppose that $ u_{\rho_i}(a_i,1)$ is of orthogonal type and not a factor of $\pi_t$. It is an immediate consequence of Theorem \ref{thm Atobe unitary induction} that 
\[ u_{\rho_i}(a_i,2) \rtimes \pi(\EE) = \begin{cases}
    \pi(\EE_1) \oplus \pi(\EE_2) &\text{ if }a_i=1,\\
\pi(\EE_1) \oplus \pi(\EE_2) \oplus \pi(\EE_3) &\text{ if }a_i>1,
\end{cases}\]
where $\EE_1, \EE_2, \EE_3$ are obtained from $\EE$ by inserting 
   \[ \bordermatrix{ & \half{a_i-2} &\half{a_i}   \cr 
& \oplus& \ominus  \cr 
& \ominus & \oplus  }_{\rho_i},\ \  \bordermatrix{  & \half{a_i-2} &\half{a_i}    \cr 
& \lhd & \rhd  \cr 
& \lhd & \rhd  }_{\rho_i},\ \ \bordermatrix{  & \half{a_i-2} &\half{a_i}    \cr 
& \ominus & \oplus  \cr 
& \oplus & \ominus  }_{\rho_i}, \]
respectively. In particular, the above parabolic induction is reducible. 
  
Now we take any discrete automorphic representation $\Pi'$ of $\SO_{2N+1}(\A_k)$ in the global Arthur packet of $\psi$ such that 
\[ \Pi_v'= \bigtimes_{i=1}^s u_{\rho_i}(a_i,2) \lvert \cdot \rvert^{x_i} \rtimes \pi(\EE).\]
Such a $\Pi'$ exists by \cite[Theorem 1.5.2]{Art13}.
{Indeed, the global component group attached to the global Arthur parameter $\psi$ is trivial. Thus, for any place $v'\neq v$, we may choose $\pi_{v'}\in\Pi_{\phi_{\psi_{v'}}}$ corresponding to the trivial character. Then $\pi=\otimes'_{v}\pi_{v}$ is a discrete automorphic representation lying in the global Arthur packet of $\psi$  by \cite[Theorem 1.5.2]{Art13}.}
As $\Pi_v'$ is unitary, Theorem \ref{thm A+,u} implies that the set 
\[\{ 1 \leq j \leq s\ | \ \rho_j \cong \rho_i \text{ and }a_j=a_i\}\]
must have even cardinality. This completes the proof of this case.

Next, we deal with the case that $N=2n$ is even and $\Pi$ is of symplectic type. This time we look at $\psi= \Pi\boxtimes 1$, a global Arthur parameter of $\SO_{2n+1}(\A_k)$. For simplicity of notation, we again assume that the local Arthur parameter 
\[ (\psi_t)_v= \bigoplus_{i=s+1}^r \rho_i \otimes S_{a_i}\otimes S_1\]
is of good parity. 

Let $I=\{s+1,\ldots, r\}$ and assume that $a_i$ is non-decreasing. Let $I_{\rho}= \{i \in I \ | \ \rho_i \cong \rho\}$. Then any representation in the packet $\Pi_{(\psi_t)_v}$ is of the form $\pi(\EE)$ for some
\[ \EE= \bigcup_{\rho} \{ ([\half{a_i-1}, \half{a_i-1}]_{\rho}, 0, \eta_i)\}_{i \in I_{\rho}}.\]

Now suppose that $ u_{\rho_i}(a_i,1)$ is of symplectic type and not a factor of $\pi_t$. Then
\[ u_{\rho_i}(a_i,1) \rtimes \pi(\EE) = \pi(\EE_1) \oplus \pi(\EE_2),\]
where $\EE_1, \EE_2$ are obtained from $\EE$ by inserting
    \begin{align*}
        \bordermatrix{ & \half{a_i-1}    \cr 
& \oplus\cr 
&\oplus  }_{\rho_i},\ \ 
\bordermatrix{ & \half{a_i-1}   \cr 
& \ominus  \cr 
& \ominus  }_{\rho_i},
    \end{align*}
respectively. Then the same argument in the previous case (for example, taking all $\eta_i=1$ then globalizing) shows that the set  
\[\{ 1 \leq j \leq s\ | \ \rho_j \cong \rho_i \text{ and }a_j=a_i\}\]
must have even multiplicity. This completes the proof of the theorem.
\end{proof}

Next, we consider several lower rank cases. 

\subsection{The case of \texorpdfstring{$N=2$}{}}

Let $\Pi$ be an irreducible self-dual cuspidal automorphic representation of $\GL_2(\A_k)$.

     Since $\Pi_v$ is a unitary representation of $\GL_2(k_v),$ by the classification of the unitary dual of $\GL_2(k_v)$ in \cite{Tad86}, there are four possible cases of $\Pi_{v}$: Here $\chi, \chi_1, \chi_2$ are unitary characters of $k_v^{\times}$.
\begin{enumerate}
    \item [(i)] $\Pi_v= u_{\chi}(2,1)$, or
    \item [(ii)] $\Pi_v= u_{\chi}(1,2)$, or
    \item [(iii)] $\Pi_{v}= \chi \lvert \cdot \rvert^{x} \times \chi\lvert \cdot \rvert^{-x}$, where $0 <x<\half{1}$,
    \item [(iv)] $\Pi_v= \chi_1 \times \chi_2$, or
    \item [(v)] $\Pi_v=\rho$ is a unitary supercuspidal representation of $\GL_2(k_v)$.
\end{enumerate}

Assume that $\Pi$ is of orthogonal type. 
Cases (i) and (v) are already tempered. Case (ii) is impossible since it is not generic. Also, in Case (iv), we must have $\chi_2= \chi_1^{\vee}$ in order that the $L$-parameter of $\Pi_v$ factors through the split even orthogonal group, hence it is tempered. 
Now, we consider case (iii). 
Let $\phi_\chi$ be the local $L$-parameter for $\chi.$ Then the $L$-parameter of $\Pi_v$ is 
\[
\phi=\phi_\chi\lvert \cdot \rvert^{x}+ \phi_\chi\lvert \cdot \rvert^{-x}.
\]
Since $\Pi_v$ is of orthogonal type, this $L$-parameter must factor through $\RO_{2}(\BC)$. Therefore, we must have $\chi \cong \chi^{\vee}=\chi^{-1}$, or equivalently, $\chi^2=1$, i.e. $\chi$ is quadratic. This contradicts Theorem \ref{thm Ramanujan 1 intro}. We conclude that Case (iii) is also impossible.

Therefore, we can make the conclusion as follows: 
{\it 
    Suppose that $\Pi$ is an irreducible cuspidal automorphic representation of $\GL_2(\mathbb{A}_k)$ of orthogonal type. 
    Then for any local place $v < \infty$, $\Pi_v$ is tempered.}

We remark that when $\Pi$ is of symplectic type, Theorem \ref{thm Ramanujan 1 intro} cannot be applied to rule out case (iii).

\subsection{The case of \texorpdfstring{$N=3$}{}} 
 Suppose that $\Pi$ is an irreducible self-dual cuspidal automorphic representation of $\GL_3(\mathbb{A}_k)$. Then $\Pi$ is automatically of orthogonal type. For any $v < \infty$, $\Pi_v$ is generic. From the classification of the unitary dual for $\GL_3(k_v)$, it is of one of the following forms. 
\begin{itemize}
    \item [(i)]  $ \rho_3$.
    \item [(ii)] $ \rho_2 \times \chi$.
    \item [(iii)] $u_{\chi}(3,1)$.
    \item [(iv)] $u_{\chi_1}(2,1) \times \chi_2$.
    \item [(v)] $\chi_1 \times \chi_2 \times \chi_3$.
    \item [(vi)] $\chi \lvert \cdot \rvert^x \times 1 \times \chi \lvert \cdot \rvert^{-x},$ where  $\chi^2=1$ and $0 <  x<\half{1}$.
\end{itemize}
Here $\chi, \chi_1,\chi_2, \chi_3$ are unitary characters of $k_v^{\times}$, and $\rho_3, \rho_2$ are unitary supercuspidal representations of $\GL_3(k_v)$ and $\GL_2(k_v)$ respectively. 
 
 For any $v < \infty$, assume
we are in case (vi). Applying Theorem \ref{thm Ramanujan 1 intro}, we can conclude that in Case (vi) we must have $\chi=1$.
In particular, if $\Pi_v$ is ramified, then $\Pi_v$ is tempered.

\subsection{More examples}

    One can apply the same strategy for $\GL_N(\A_k)$ for $N >3$. However, the list becomes more complicated as $N$ increases. For example, for $N=4$, the possible non-tempered local components of a self-dual cuspidal automorphic representation are of one of the following forms: 
\begin{enumerate}
\item [(i)] $\rho\lvert \cdot \rvert^{x} \times \rho\lvert \cdot \rvert^{-x} $.
    \item [(ii)] $ u_{\chi}(2,1)\lvert \cdot \rvert^{x} \times u_{\chi}(2,1)\lvert \cdot \rvert^{-x}$.
    \item [(iii)] $\chi\lvert \cdot \rvert^{x} \times \rho \times \chi\lvert \cdot \rvert^{-x}$.
    \item [(iv)] $\chi_1\lvert \cdot \rvert^{x_1} \times u_{\chi_2}(2,1) \times \chi_1\lvert \cdot \rvert^{-x_1}$.
    \item [(v)] $\chi_1\lvert \cdot \rvert^{x_1} \times \chi_2 \times \chi_3 \times \chi_1\lvert \cdot \rvert^{-x_1}$.
    \item [(vi)] $\chi_1\lvert \cdot \rvert^{x_1} \times \chi_2 \lvert \cdot \rvert^{x_2} \times \chi_2\lvert \cdot \rvert^{-x_2} \times \chi_1\lvert \cdot \rvert^{-x_1}$.
\end{enumerate}
Here $\chi, \chi_1,\chi_2, \chi_3$ are unitary characters of $\GL_1(k_v)$, $\rho$ is a unitary supercuspidal representation of $\GL_2(k_v)$ and $0 \leq x,x_1,x_2<\half{1}$. Since the local representation is self-dual, we have $\rho$ and $\chi$ are self-dual, $\chi_1$ is self-dual if $x_1 \neq 0$ in Cases (iv), (v). Finally, in Case (vi), $\chi_1$ and $\chi_2$ are self-dual unless $\chi_1= \chi_2^{\vee}$ and $x_2 =\pm x_1$.  

Let $\Pi$ be an irreducible self-dual cuspidal automorphic representation of $\GL_4(\A_k)$. If $\Pi$ is of orthogonal type, applying Theorem \ref{thm Ramanujan 1 intro}, we can conclude that the exponents will be zero in the following cases:
\begin{itemize}
    \item Case (i) when $\rho$ is of  orthogonal type;
    \item  Case (iii);
    \item Case (v) when $\chi_1 \neq \chi_2$ and  $\chi_1 \neq \chi_3$;
    \item Case (vi) when $\chi_1,\chi_2$ are self-dual and  $\chi_1 \neq \chi_2$.
\end{itemize}
Note that Case (iv) does not happen since the local representation of $\GL_{4}(k_v)$ is not of orthogonal type. If $\Pi$ is of symplectic type, we can conclude that the exponents will be zero in the following cases:
\begin{itemize}
    \item Case (i) when $\rho$ is of  symplectic type;
    \item  Case (ii).
\end{itemize}

\end{document}